\documentclass[11pt,a4paper]{article}%
\usepackage{amscd}
\usepackage{amsmath}
\usepackage{graphicx}
\usepackage{amsfonts}
\usepackage{amssymb}
\usepackage{xcolor}
\usepackage{amsbsy,amsthm}
\usepackage{pifont}
\usepackage{fancyhdr,ifthen}
\usepackage[colorlinks=true, bookmarksnumbered=true, bookmarksopen=true,
bookmarksopenlevel=3, pdfstartview=FitH, linkcolor=cyan, pdfmenubar=true,
pdftoolbar=true, bookmarks=true,citecolor=cyan, urlcolor=magenta,
filecolor=cyan,menucolor=black,plainpages=false,pdfpagelabels]{hyperref}%
\DeclareMathOperator{\dv}{div}

\newtheorem{theorem}{Theorem}[section]
\newtheorem{lemma}[theorem]{Lemma}

\newtheorem{proposition}[theorem]{Proposition}
\newtheorem{definition}[theorem]{Definition}
\newtheorem{remark}[theorem]{Remark}

\numberwithin{equation}{section}

\begin{document}

\title{On product estimates in Besov-Morrey spaces and stationary solutions for the
Hall-MHD and Navier-Stokes systems\vspace{0.7cm} }
\author{{\Large {Lucas C. F. Ferreira\thanks{State University of Campinas (Unicamp),
IMECC-Department of Mathematics, Rua S\'{e}rgio Buarque de Holanda, CEP
13083-859, Campinas, SP, Brazil. Email:\ lcff@ime.unicamp.br (corresponding
author).} \ \ and \ Rafael P. da Silva\thanks{Federal University of Technology
- Paran\'{a} (UTFPR), Department of Mathematics, CEP 86300-000, Corn\'{e}lio
Proc\'{o}pio-PR, Brazil. Email:\ rpsilva@utfpr.edu.br.}}}\vspace{0.5cm}}
\date{}
\maketitle

\begin{abstract}
In this paper, we investigate the well-posedness of the three-dimensional
stationary incompressible resistive-viscous Hall magnetohydrodynamic
(Hall-MHD) system within the framework of scale-invariant Besov-Morrey spaces.
A key component of our analysis is the derivation of product estimates in
these spaces, which are instrumental in handling the system's nonlinear terms.
As an application, we also establish a well-posedness result for the
stationary Navier-Stokes equations in this setting, thereby broadening the
admissible classes of external forces and solutions.



\

\end{abstract}

\pagestyle{fancy} \fancyhf{} \renewcommand{\headrulewidth}{0pt}
\chead{\ifthenelse{\isodd{\value{page}}}{Lucas C. F. Ferreira and Rafael P. da Silva}{Stationary
Hall-MHD system and Navier-Stokes equations}} \rhead{\thepage}

\section{Introduction}

We are chiefly centered on examining the 3D stationary
Hall-magnetohydrodynamics (Hall-MHD) system
\begin{equation}
\left\{
\begin{array}
[c]{rclll}%
(u\cdot\nabla)u-\mu\Delta u-(\nabla\times B)\times B+\nabla\phi & = & f_{1}, &
\text{in} & \,\,\mathbb{R}^{3};\\
-\nu\Delta B-\nabla\times((u-h(\nabla\times B))\times B) & = & f_{2}, &
\text{in} & \,\,\mathbb{R}^{3};\\
\dv u=\dv B & = & 0, & \text{in} & \,\,\mathbb{R}^{3}.
\end{array}
\right.  \label{MHD}%
\end{equation}
Here, the unknown vector fields $u=u(x)$ and $B=B(x)$ denote the velocity and
magnetic fields, respectively, while $\phi=\phi(x)$ is the scalar pressure
function. Additionally, the fields $f_{1}=f_{1}(x)$ and $f_{2}=f_{2}(x)$
represent given external forces. The positive constants $\mu,\nu$ correspond
to the dynamic viscosity and the magnetic resistivity, respectively. Moreover,
the dimensionless parameter $h$, also positive, measures the Hall effect
intensity in the fluid.

Modeling electrically conducting fluids leads to the formulation of system
(\ref{MHD}), which is applicable to diverse physical scenarios, such as
geo-dynamo processes, neutron stars, and plasma magnetic reconnection (refer
to \cite{Acher-Liu}, \cite{Lighthill} and additional sources cited therein).
When juxtaposed with the standard MHD system (see \cite{Sermange}), our system
incorporates the Hall electric field $hJ\times B$, where the current $J$ is
defined as $\nabla\times B$. This, in turn, generates the Hall term
$h\nabla\times((\nabla\times B)\times B)$ present in (\ref{MHD})$_{2}$, which
introduces an inherent difficulty in the direct application of techniques
established for the study of Navier-Stokes equations, classical MHD equations,
and other related fluid dynamics models, a fact that is due to the
specificities of their mathematical structures and restrictions. Readers
seeking further insights into the physical background of (\ref{MHD}) are
directed to \cite{Balbus-Terquem, Forbes, Huba, Mininni, Shalybkov-Urpin,
Wardle}.

Regarding the non-stationary Hall-MHD system, we have a starting point with
the work \cite{Acher-Liu} developed by Acheritogaray \textit{et al.}, in which
they presented a physical-mathematical deduction of the Hall-MHD system and
established the existence of global-in-time weak solutions in $L^{2}%
([0,1]^{3})$. From there, a series of studies have been carried out on
well-posedness, regularity, and asymptotic behavior of solutions. In
\cite{Chae-Degond-Liu}, Chae \textit{et al.} achieved results in local and
global well-posedness for initial data in $H^{s}(s>5/2)$, whereas Dumas and
Sueur \cite{Dumas-Sueur} engaged in further research on weak solutions. The
global existence of strong solutions for sufficiently small initial data and
Serrin-type continuation criteria was demonstrated by Chae and Lee
\cite{Chae-Lee} and Ye \cite{Ye}, while Ahmad \textit{et al. }%
\cite{Ahmad-Zhou} indicated that the velocity can be used to control the
potential blow-up of smooth solutions. Furthermore, Benvenutti and Ferreira
\cite{Benvenutti-Ferreira}, Wu \textit{et al.} \cite{Wu-Yu-Tang}, and Wan and
Zhou \cite{Wan-Zhou} provided results on the existence and uniqueness of
strong solutions with less regularity initial data in Sobolev or Besov spaces.
In more recent research, Danchin and Tan \cite{Danchin-Tan, Danchin-Tan-2}
worked in the framework of critical Sobolev and Besov spaces, having obtained
local and global-in-time well-posedness results. Furthermore, Ferreira and da
Silva \cite{Ferreira-da Silva} extended the analysis of the non-stationary
Hall-MHD system to Besov-Morrey spaces by proving well-posedness results.
Several other papers have also contributed to this topic, including
\cite{An-Chen-Han, Fujii, Han-Hu-Lai, Kawashima-Nakasato-Ogawa, Li-Yu-Zhu,
Liu-Tan, Nakasato, Wan-Zhou-2, Zhang}.

When it comes to the stationary system (\ref{MHD}), what we predominantly find
in the literature are Liouville type theorems. In this context, we can mention
the works \cite{Chae-Kim-Wolf, Chae-Wolf, Cho-Neustupa-Yang, Li-Niu, Li-Su,
Liu, Wang-Yang, Yuan-Xiao}. In contrast, relatively few studies seem to
address the problem of existence and uniqueness of solutions with given
external forces. With respect to this approach, we have the recent paper
\cite{Tan-Tsurumi-Zhang} by Tan \textit{et al.}, where a well-posedness result
was established for external forces $f_{1}\in\dot{B}_{p,1}^{\frac{3}{p}-3}$
and $f_{2}\in\dot{B}_{p,r}^{\frac{3}{p}-3}\cap\dot{B}_{p,1}^{\frac{3}{p}%
-2},r\in\lbrack1,2]$.

Against this background, the present paper primarily aims to analyze the
stationary Hall-MHD system within a new functional framework known as
homogeneous Besov-Morrey spaces (BM-spaces), denoted by $\mathcal{N}%
_{p,q,r}^{s}$. Pioneered by Kozono and Yamazaki \cite{Kozono-Yamazaki} in the
1990s, these spaces fall under the class of Besov spaces, with Morrey spaces
$\mathcal{M}_{q}^{p}$ acting as the underlying base. They provide an effective
extension that integrates key features of both Besov and Morrey spaces,
allowing for a more comprehensive framework that encompasses the regularity
properties inherent in Besov spaces alongside the structural features of
Morrey spaces. Initially employed in \cite{Kozono-Yamazaki} for the analysis
of the Navier-Stokes equations via the Kato approach (see also
\cite{Mazzucato}), these spaces have significantly contributed to the study of
other fluid dynamics models and PDEs, as highlighted in sources like
\cite{Bie-Wang-Yao, Duarte-Ferreira-Villa, Ferreira-Perez, Ferreira-Postigo,
Xu-Tan, Yang-Fu-Sun}.

Specifically, we establish a well-posedness result for (\ref{MHD}) in critical
BM-spaces, with external forces $f_{1}$ and $f_{2}$ belonging to
$\mathcal{N}_{p,q,1}^{\frac{3}{p}-3}$ and $\mathcal{N}_{p,q,r}^{\frac{3}{p}%
-3}\cap\mathcal{N}_{p,q,1}^{\frac{3}{p}-2}$($1\leq r\leq2$), respectively.
Given that $\mathcal{N}_{p,q,r}^{s}$ is strictly larger than the standard
homogeneous Besov space $\dot{B}_{p,r}^{s}$ when $q<p$, $r\in\lbrack1,\infty]$
and $s\in\mathbb{R},$ this enables us to incorporate a broader range of
external forces.

Differing from classical stationary MHD ($h=0$), the system in question lacks
scaling invariance as a result of the coexistence of the Lorentz force in
(\ref{MHD})$_{1}$ and the Hall term in (\ref{MHD})$_{2}$. The simultaneous
presence of these factors demands an accurate comprehension of critical
regularity. To address this, we follow the approach proposed by Danchin and
Tan \cite{Danchin-Tan-2} for the non-stationary case, decoupling the system by
taking $u\equiv0$ and $B\equiv0$ in (\ref{MHD}). This leads us to the
stationary incompressible Navier-Stokes equations and the stationary Hall equations%

\begin{equation}
\left\{
\begin{array}
[c]{rcll}%
-\mu\Delta u+(u\cdot\nabla)u+\nabla\phi & = & f, & \\
\dv u & = & 0, &
\end{array}
\right.  \label{aux-NS100}%
\end{equation}
and
\begin{equation}
-\nu\Delta B+h\nabla\times((\nabla\times B)\times B)=g, \label{aux-Hall100}%
\end{equation}
where here, for the sake of simplicity, $f_{1}$ and $f_{2}$ were replaced by
$f$ and $g$, respectively. These problems have the respective scaling maps
\[
(u,\phi,f)(x)\rightarrow(\lambda u,\lambda^{2}\phi,\lambda^{3}f)(\lambda
x)\text{ and }(B,g)(x)\rightarrow(B,\lambda^{2}g)(\lambda x).
\]
Therefore, for $h>0$, one might naturally expect $u$ and $\nabla B$ to have
the same level of homogeneity and, consequently, the same regularity. However,
in the classical MHD system ($h=0$)$,$ both $u$ and $B$ share the same homogeneity.

To ensure compatibility in terms of scaling, Danchin and Tan
\cite{Danchin-Tan-2} introduced the current function $J=\nabla\times B$. A
fact that, combined with the condition $\dv B=0,$ results in $\Delta
B=-(\nabla\times J)$. So, taking the operator $curl^{-1}:= (-\Delta
)^{-1}\nabla\,\times$, which can be regarded as the Fourier multiplier
operator with symbol $i|\xi|^{-2}\xi\,\times$, the field $B$ can be expressed as%

\begin{equation}
B=(-\Delta)^{-1}\nabla\times J=curl^{-1}J. \label{B-identity}%
\end{equation}
As a result, one arrives at the so-called extended stationary Hall-MHD system%

\begin{align}
(u\cdot\nabla)u-\mu\Delta u-J\times B+\nabla\phi &  =f_{1},\label{Ext.HMHD1}\\
-\nabla\times((u-hJ)\times B)-\nu\Delta B  &  =f_{2},\label{Ext.HMHD2}\\
-\nabla\times(\nabla\times((u-hJ)\times curl^{-1}J))-\nu\Delta J  &
=\nabla\times f_{2},\label{Ext.HMHD3}\\
\dv u=\dv B  &  =0, \label{Ext.HMHD4}%
\end{align}
which exhibits the scaling-invariant property
\begin{equation}
(u,B,J)(x)\rightarrow\lambda(u,B,J)(\lambda x),\,\,\,\phi(x)\rightarrow
\lambda^{2}\phi(\lambda x),\,\,\,(f_{1},f_{2},\nabla\times f_{2}%
)(x)\rightarrow\lambda^{3}(f_{1},f_{2},\nabla\times f_{2})(\lambda x),
\label{scal-ExHMHD-1}%
\end{equation}
providing the advantage of aligning with the homogeneity of the stationary
incompressible Navier-Stokes equations (\ref{aux-NS100}).

To conduct an existence-uniqueness analysis for (\ref{Ext.HMHD1}%
)-(\ref{Ext.HMHD4}) in the framework of critical BM-spaces, which will allow
to achieve the desired results for system (\ref{MHD}), it is essential to
establish suitable product estimates that enable us to handle the
corresponding non-linear terms. Motivated by this, our first aim is to prove
the product estimates outlined below, as we have not been able to find in the
literature estimates that fully meet our needs (see further details in Remark
\ref{product estimates relevance}). From now on, when dealing with contexts
that extend to arbitrary dimensions, we denote the dimension of the spatial
variable by $n\in\mathbb{N}$.

\begin{proposition}
\label{Product:3/p-1} Let $n\geq2,$ $1\leq q\leq p<n,$ $1\leq r\leq\infty,$
and $s\in\mathbb{R}$. Then, $\exists C=C(p,q,n)>0$ such that
\[
\Vert uv\Vert_{\mathcal{N}_{p,q,r}^{\frac{n}{p}-1}}\leq C\Vert u\Vert
_{\mathcal{N}_{p,q,r}^{\frac{n}{p}-1}}\Vert v\Vert_{\mathcal{N}_{p,q,\infty
}^{\frac{n}{p}}\cap L^{\infty}},
\]
for all $u\in\mathcal{N}_{p,q,r}^{\frac{n}{p}-1}$ and $v\in\mathcal{N}%
_{p,q,\infty}^{\frac{n}{p}}\cap L^{\infty}$. In particular, for $u\in
\mathcal{N}_{p,q,r}^{\frac{n}{p}-1}$ and $v\in\mathcal{N}_{p,q,1}^{\frac{n}%
{p}}$, we have
\[
\Vert uv\Vert_{\mathcal{N}_{p,q,r}^{\frac{n}{p}-1}}\leq\tilde{C}\Vert
u\Vert_{\mathcal{N}_{p,q,r}^{\frac{n}{p}-1}}\Vert v\Vert_{\mathcal{N}%
_{p,q,1}^{\frac{n}{p}}},
\]
for some constant $\tilde{C}=\tilde{C}(p,q,n)>0$.
\end{proposition}

\begin{proposition}
\label{Product:3/p-2} Let $n\geq3,1\leq q\leq p<\infty,$ and $1\leq
r,r_{1},r_{2}\leq\infty$ satisfying $\displaystyle\frac{1}{r}%
=\displaystyle\frac{1}{r_{1}}+\displaystyle\frac{1}{r_{2}}$. If

\begin{description}
\item[(i)] {$p<\frac{n}{2}$}
\end{description}

or

\begin{description}
\item[(ii)] {$\frac{n}{2}\leq p<n$ and $q\geq\frac{3p+n}{2n}$,}
\end{description}

then there exists a constant $C=C(p,q,n)>0$ such that
\[
\Vert uv\Vert_{\mathcal{N}_{p,q,r}^{\frac{n}{p}-2}}\leq C\Vert u\Vert
_{\mathcal{N}_{p,q,r_{1}}^{\frac{n}{p}-1}}\Vert v\Vert_{\mathcal{N}%
_{p,q,r_{2}}^{\frac{n}{p}-1}},
\]
for all $u\in\mathcal{N}_{p,q,r_{1}}^{\frac{n}{p}-1}$ and $v\in\mathcal{N}%
_{p,q,r_{2}}^{\frac{n}{p}-1}$.
\end{proposition}

\begin{remark}
\label{hypothesis-compatibility}Note that, in Proposition \ref{Product:3/p-1},
the condition $n\geq2$ is necessary to the compatibility of the inequality
$1\leq q\leq p<n$, which ensures that $\frac{n}{p}-1>0$, a fact that plays a
fundamental role in the proof of this result. Moreover, in Proposition
\ref{Product:3/p-2}, for $n\geq3$ and $\frac{n}{2}\leq p<n$, we have
$\frac{3p+n}{2n}\leq\frac{5p}{6}<p$. Therefore, the hypothesis $q\geq
\frac{3p+n}{2n}$ does not contradict the one $q\leq p$.
\end{remark}

\begin{remark}
\label{product estimates relevance} The estimates presented in the
propositions above are basic tools for attaining our goals, since Proposition
\ref{Product:3/p-2} allows us to handle nonlinearities of Navier-Stokes type,
and Proposition \ref{Product:3/p-1} makes us able to work with the Hall term
$h\nabla\times((\nabla\times B)\times B)$. In our search for product estimates
in the framework of Besov-Morrey spaces that would assist in solving our
problem, the closest we found was \cite[Proposition 1.1]%
{Ferreira-Perez-Villamizar}, where the authors established a result that
encompasses the content of our Proposition \ref{Product:3/p-2} for $\frac
{n}{2}<p<n$, $r_{1}=\infty$, and $r_{2}=r$. Considering that, in the context
of this comparison, we cover a wider range of parameters and, moreover, the
aforementioned result from \cite{Ferreira-Perez-Villamizar} does not include
the content of Proposition \ref{Product:3/p-1}, we emphasize the need for
obtaining the product estimates stated above.
\end{remark}

Throughout our research, we identified that Proposition \ref{Product:3/p-2},
in addition to being a fundamental tool for achieving our goal with regard to
the stationary Hall-MHD system, also enables us to develop a well-posedness
analysis for the stationary Navier-Stokes equations. Concerning problem
(\ref{aux-NS100}) in the whole space $\mathbb{R}^{n}$, the literature includes
the paper by Chen \cite{Chen}, in which an existence and uniqueness result was
established within the setting of Sobolev spaces. In \cite{Kozono-Yamazaki-2},
Kozono and Yamazaki worked in Morrey spaces, where they conducted a stability
analysis of the stationary solution alongside a well-posedness study.
Subsequently, Kaneko \textit{et. al.} \cite{Kaneko-Kozono-Shimizu} showed
existence, uniqueness, and regularity of solutions in critical homogeneous
Besov spaces. Moreover, Ferreira \textit{et al.}
\cite{Ferreira-Perez-Villamizar} proved, in a scaling-invariant setting of
Besov-Lorentz-Morrey spaces (which include BM-spaces), an existence and
uniqueness result for the stationary Boussinesq equations, consequently
providing, for the Navier-Stokes equations, a new class of stationary solutions.

Based on the previous discussion, the contribution to this topic in our paper
will be to establish a well-posedness result for problem (\ref{aux-NS100}) in
the framework of critical BM-spaces. The goal is to extend the class of
solutions obtained in \cite{Kaneko-Kozono-Shimizu} and complement, in the
scope of stationary Navier-Stokes equations and Besov-Morrey spaces, the
result established in \cite{Ferreira-Perez-Villamizar}.

Prior to stating our main results, let us introduce some functional spaces
that will be used in the development of our study. In this sense, we define
the Banach spaces
\begin{equation}
X=\left\{  (u,B)\in\mathcal{N}_{p,q,1}^{\frac{3}{p}-1}\times\mathcal{N}%
_{p,q,r}^{\frac{3}{p}-1}:\nabla\times B\in\mathcal{N}_{p,q,1}^{\frac{3}{p}%
-1}\,\,\text{and }\dv u=\dv B=0\text{ in}\,\,\mathcal{S}^{^{\prime}%
}(\mathbb{R}^{3})\right\}  , \label{X-space}%
\end{equation}
with norm
\[
\Vert(u,B)\Vert_{X}=\Vert u\Vert_{\mathcal{N}_{p,q,1}^{\frac{3}{p}-1}}+\Vert
B\Vert_{\mathcal{N}_{p,q,r}^{\frac{3}{p}-1}}+\Vert\nabla\times B\Vert
_{\mathcal{N}_{p,q,1}^{\frac{3}{p}-1}},
\]
and
\begin{equation}
Y=\left\{  (f_{1},f_{2})\in\mathcal{N}_{p,q,1}^{\frac{3}{p}-3}\times
\mathcal{N}_{p,q,r}^{\frac{3}{p}-3}:\nabla\times f_{2}\in\mathcal{N}%
_{p,q,1}^{\frac{3}{p}-3}\,\,\text{and }\dv f_{2}=0\text{ in}\,\,\mathcal{S}%
^{^{\prime}}(\mathbb{R}^{3})\right\}  , \label{Y-space}%
\end{equation}
with norm
\[
\Vert(f_{1},f_{2})\Vert_{Y}=\Vert f_{1}\Vert_{\mathcal{N}_{p,q,1}^{\frac{3}%
{p}-3}}+\Vert f_{2}\Vert_{\mathcal{N}_{p,q,r}^{\frac{3}{p}-3}}+\Vert
\nabla\times f_{2}\Vert_{\mathcal{N}_{p,q,1}^{\frac{3}{p}-3}}.
\]
We are now ready to present our main results. The precise sense in which we
use the term well-posedness in its statements can be found in Section
\ref{Sec2}, see Definition \ref{well-posed}.

\begin{theorem}
\label{Result1} Let $1\leq q\leq p<\infty,$ and $1\leq r\leq2$. If either
$p<\frac{3}{2}$ or $\frac{3}{2}\leq p<3$ and $q\geq\frac{p+1}{2}$, then the
stationary Hall-MHD system (\ref{MHD}) is well-posed from $Y$ to $X$.
\end{theorem}

\begin{theorem}
\label{Result2} Let $n\geq3$, $1\leq q\leq p<\infty,$ and $1\leq r\leq\infty$.
If either $p<\frac{n}{2}$ or $\frac{n}{2}\leq p<n$ and $q\geq\frac{3p+n}{2n}$,
then the stationary Navier-Stokes equation (\ref{aux-NS100}) in $\mathbb{R}%
^{n}$ is well-posed from $\mathcal{N}_{p,q,r}^{\frac{n}{p}-3}$ to
$\mathcal{N}_{p,q,r}^{\frac{n}{p}-1}$.
\end{theorem}

\begin{remark}
\label{paper-relevance-1} In \cite{Tan-Tsurumi-Zhang}, the authors established
well-posedness for the stationary Hall-MHD system, under the conditions $1\leq
p<3$ and $1\leq r\leq2$, from the data-space
\[
\left\{  (f_{1},f_{2})\in\dot{B}_{p,1}^{\frac{3}{p}-3}\times\dot{B}%
_{p,r}^{\frac{3}{p}-3}:\nabla\times f_{2}\in\dot{B}_{p,1}^{\frac{3}{p}%
-3}\,\,\text{and }\dv f_{2}=0\text{ in}\,\,\mathcal{S}^{^{\prime}}%
(\mathbb{R}^{3})\right\}
\]
to the solution-space
\[
\left\{  (u,B)\in\dot{B}_{p,1}^{\frac{3}{p}-1}\times\dot{B}_{p,r}^{\frac{3}%
{p}-1}:\nabla\times B\in\dot{B}_{p,1}^{\frac{3}{p}-1}\,\,\text{and }\dv u=\dv
B=0\text{ in}\,\,\mathcal{S}^{^{\prime}}(\mathbb{R}^{3})\right\}  .
\]
Since $\mathcal{N}_{p,p,r}^{s}=\dot{B}_{p,r}^{s}$ and, as reported by \cite[p.
964 and Example 0.5 on p. 966]{Kozono-Yamazaki} and \cite[p. 1307, lines 1 to
6]{Mazzucato}, for $q<p,r\in\lbrack1,\infty]$, and $s\in\mathbb{R}$ we have
the strict inclusion $\dot{B}_{p,r}^{s}\hookrightarrow\mathcal{N}_{p,q,r}^{s}%
$, Theorem \ref{Result1} not only covers this result but also extends it,
providing a larger class of stationary solutions for the Hall-MHD system.
\end{remark}

\begin{remark}
\label{paper-relevance-2} In \cite{Kaneko-Kozono-Shimizu} the well-posedness
of the stationary Navier-Stokes equations in $\mathbb{R}^{n}$ was established
for $n\geq3$ and $1\leq p<n$, from $\dot{B}_{p,r}^{\frac{n}{p}-3}$ to $\dot
{B}_{p,r}^{\frac{n}{p}-1}$, where $1\leq r\leq\infty$. Furthermore, as a
consequence of the work carried out in \cite{Ferreira-Perez-Villamizar}, we
have the existence and uniqueness of stationary solutions for the
Navier-Stokes equations in the framework of critical BM-spaces, given that
$n\geq3$ and $\frac{n}{2}<p<n$. Thus, taking into account the relationships
between the Besov spaces and Besov-Morrey spaces highlighted in Remark
\ref{paper-relevance-1}, we have that Theorem \ref{Result2} covers and extends
the aforementioned result of \cite{Kaneko-Kozono-Shimizu} by contemplating a
larger class of external forces. Moreover, from the perspective of the
Navier-Stokes equations and Besov-Morrey spaces, it complements the result
obtained in \cite{Ferreira-Perez-Villamizar}.
\end{remark}

As for the structure of our paper, Section \ref{Sec2} introduces key notation
and revisits definitions and features of function spaces such as Morrey and
Besov-Morrey spaces, which will be regularly used in our work. Section
\ref{Sec3} is devoted to proving the product estimates given in Propositions
\ref{Product:3/p-1} and \ref{Product:3/p-2}. In Section \ref{Sec4}, we
concentrate on the analysis of the stationary Hall-MHD system and establish
the proof of Theorem \ref{Result1}. Finally, in Section \ref{Sec5}, we turn
our attention to the stationary Navier-Stokes equations and prove Theorem
\ref{Result2}.

\section{Preliminaries}

\label{Sec2}In this section, we will gather and outline essential notations,
definitions, tools, and properties pertinent to certain functional spaces and
operators that are relevant to our study. Here the dimension $n\geq1$ is
assumed. Throughout the paper, the letter $C>0$ will be used to denote a
constant that may vary with each instance.

Our initial focus will be on an overview of Morrey spaces, addressing their
definition and key properties. For further information, please refer to
\cite{Kozono-Yamazaki, Rosenthal, Sawano-Fazio-Hakim, Zhou}.

\begin{definition}
Consider $1\leq q\leq p<\infty$. The Morrey space $\mathcal{M}_{q}%
^{p}=\mathcal{M}_{q}^{p}(\mathbb{R}^{n})$ is defined as the set of all
functions $u\in L_{loc}^{q}(\mathbb{R}^{n})$ such that%

\[
\left\Vert u\right\Vert _{\mathcal{M}_{q}^{p}}:=\sup_{x_{0}\in\mathbb{R}%
^{n},R>0}R^{n/p-n/q}\left\{  \int_{B_{R}(x_{0})}|u|^{q}dx\right\}  ^{\frac
{1}{q}}<\infty,
\]
where $B_{R}(x_{0})\subset\mathbb{R}^{n}$ denotes the closed ball centered in
$x_{0}$ with radius $R>0$.
\end{definition}

\begin{remark}
It is straightforward to check that $\mathcal{M}_{q}^{p}$, endowed with the
norm $\Vert\ .\ \Vert_{\mathcal{M}_{q}^{p}}$, is a Banach space. Moreover, for
$1\leq q<q^{^{\prime}}\leq p<\infty$, the continuous inclusion relation
$\mathcal{M}_{q^{^{\prime}}}^{p}\subset\mathcal{M}_{q}^{p}$ holds and
$\mathcal{M}_{p}^{p}=L^{p}(\mathbb{R}^{n})$. When $p=\infty, \mathcal{M}%
_{q}^{p}$ can be identified with $L^{\infty}(\mathbb{R}^{n})$, meaning
$\mathcal{M}_{q}^{\infty}=L^{\infty}$.
\end{remark}

The following results, which are available in \cite{Kozono-Yamazaki}, presents
H\"{o}lder-type and Young inequalities in Morrey spaces and will be highly
significant for the advancement of our work.

\begin{lemma}
\label{Holder-Morrey} Let $1\leq q\leq p<\infty$ and $1\leq q_{j}\leq
p_{j}<\infty$, $j=1,2$. Then:

\begin{itemize}
\item[(i)] {For all $u_{0}\in L^{\infty}$ and $u\in\mathcal{M}_{q}^{p}$, we
have that
\[
\Vert u_{0}u\Vert_{\mathcal{M}_{q}^{p}}\leq\Vert u_{0}\Vert_{L^{\infty}} \Vert
u\Vert_{\mathcal{M}_{q}^{p}}.
\]
}

\item[(ii)] {If $u_{j}\in\mathcal{M}_{q_{j}}^{p_{j}}, j=1,2$ and
\[
\displaystyle\frac{1}{p}=\displaystyle\frac{1}{p_{1}}+ \displaystyle\frac
{1}{p_{2}} \quad\text{and} \quad\displaystyle\frac{1}{q}\geq\displaystyle\frac
{1}{q_{1}} + \displaystyle\frac{1}{q_{2}},
\]
we have that
\[
\Vert u_{1}u_{2}\Vert_{\mathcal{M}_{q}^{p}}\leq\Vert u_{1}\Vert_{\mathcal{M}%
_{q_{1}}^{p_{1}}}\Vert u_{2}\Vert_{\mathcal{M}_{q_{2}}^{p_{2}}}.
\]
}
\end{itemize}
\end{lemma}

\begin{lemma}
\label{Young-Morrey} Let $1\leq q\leq p<\infty$, $u_{0}\in L^{1}$ and
$u\in\mathcal{M}_{q}^{p}$. Then, $u_{0}\ast u\in\mathcal{M}_{q}^{p}$ and
\[
\Vert u_{0}\ast u\Vert_{\mathcal{M}_{q}^{p}}\leq\Vert u_{0}\Vert_{L^{1}}\Vert
u\Vert_{\mathcal{M}_{q}^{p}}.
\]

\end{lemma}

Next, we review the Littlewood-Paley decomposition (further details can be
found in \cite{Kozono-Yamazaki, Lemarie, Mazzucato}). Precisely, let
$\psi,\varphi\in C_{c}^{\infty}(\mathbb{R}^{n})$ be functions meeting
$supp\varphi\subset\left\{  \xi\in\mathbb{R}^{n}:3/4\leq|\xi|\leq8/3\right\}
$, $supp\psi\subset B_{4/3}(0)$, and such that%

\[
\displaystyle\sum_{j=-\infty}^{\infty}\varphi_{j}(\xi)=1,\forall\xi
\in\mathbb{R}^{n}\backslash\{0\}\quad\mbox{and}\quad\psi(\xi
)+\displaystyle\sum_{j=0}^{\infty}\varphi_{j}(\xi)=1,\forall\xi\in
\mathbb{R}^{n},
\]
where $\varphi_{j}(\xi):=\varphi(2^{-j}\xi),\forall j\in\mathbb{Z}$ and
$B_{4/3}(0)$ denotes the ball centered at the origin with radius $4/3$.
Furthermore, for all $j\in\mathbb{Z}$, consider the homogeneous dyadic blocks
$\Delta_{j}$ defined as
\[
\Delta_{j}u := \mathcal{F}^{-1}[\varphi_{j}\mathcal{F}[u]] =\mathcal{F}%
^{-1}[\varphi_{j}]\ast u, \forall u\in\mathcal{S}^{^{\prime}}
\]
and the low-frequency cut-off operators $S_{j}$ and $\tilde{S}_{j}$, given by
\[
S_{j}u := \sum_{k\leq j}\Delta_{k}u, \forall u\in\mathcal{S}^{^{\prime}}
\]
and
\[
\tilde{S}_{j}u:=\mathcal{F}^{-1}[\psi(2^{-j}\xi)\mathcal{F}[u]]=\mathcal{F}%
^{-1}[\psi(2^{-j}\xi)]\ast u,\forall u\in\mathcal{S}^{^{\prime}},
\]
where $\mathcal{S}^{^{\prime}}=\mathcal{S}^{^{\prime}}(\mathbb{R}^{n})$
represent the space of tempered distributions. With these conditions in place,
denoting by $\mathcal{S}_{h}^{^{\prime}}=\mathcal{S}_{h}^{^{\prime}%
}(\mathbb{R}^{n}) := \left\{  u\in\mathcal{S}^{^{\prime}}:\displaystyle\lim
_{j\rightarrow-\infty}\Vert\tilde{S}_{j}u\Vert_{L^{\infty}}=0\right\}  ,$ we
have the Littlewood-Paley decomposition
\[
u=\displaystyle\sum_{j=-\infty}^{\infty}\Delta_{j}u , \forall u\in
\mathcal{S}_{h}^{^{\prime}}.
\]

\begin{remark}
\label{operators-properties} The operators defined above satisfy the
identities
\[
\Delta_{j}(\Delta_{k}u)\equiv0 \,\, \text{if} \,\, \vert j-k\vert\geq
2\quad\text{and}\quad\Delta_{j}(S_{k-2}u\Delta_{k}v)\equiv0 \,\, \text{if}
\,\, \vert j-k\vert\geq5,
\]
which can easily be checked.
\end{remark}

Let us recall two fundamental properties of operators $\Delta_{j}$ and $S_{j}%
$. The first of these is an immediate consequence of Lemma \ref{Young-Morrey},
and the second can be found in \cite{Bahouri-Chemin-Danchin}.

\begin{lemma}
\label{Delta-j.Sj.Morrey} Let $1\leq q\leq p<\infty$ and consider, for all
$j\in\mathbb{Z}$, the operators $\Delta_{j}$ and $S_{j}$ as defined above.
Then, there exist positive constants $C_{1}=C_{1}(n)$ and $C_{2}=C_{2}(n)$
such that
\[
\Vert\Delta_{j}u\Vert_{\mathcal{M}_{q}^{p}}\leq C_{1}\Vert u\Vert
_{\mathcal{M}_{q}^{p}},\forall j\in\mathbb{Z}, \forall u\in\mathcal{M}%
_{q}^{p}
\]
and
\[
\Vert S_{j}v\Vert_{\mathcal{M}_{q}^{p}}\leq C_{2}\Vert v\Vert_{\mathcal{M}%
_{q}^{p}},\forall j\in\mathbb{Z}, \forall v\in\mathcal{M}_{q}^{p}.
\]

\end{lemma}

\begin{lemma}
\label{Delta-j.Sj.Linfty} Consider, for all $j\in\mathbb{Z}$, the operators
$\Delta_{j}$ and $S_{j}$ as defined above. Then, we have the estimates
\[
\Vert\Delta_{j}u\Vert_{L^{\infty}}\leq C_{1}\Vert u\Vert_{L^{\infty}},\forall
j\in\mathbb{Z}, \forall u\in L^{\infty}
\]
and
\[
\Vert S_{j}v\Vert_{L^{\infty}}\leq C_{2}\Vert v\Vert_{L^{\infty}},\forall
j\in\mathbb{Z}, \forall v\in L^{\infty},
\]
for some constants $C_{1}=C_{1}(n)>0$ and $C_{2}=C_{2}(n)>0$.
\end{lemma}

Now, we are primed to introduce homogeneous Besov-Morrey spaces and recall
their fundamental properties (see \cite{Bie-Wang-Yao, Chen-Chen,
Kozono-Yamazaki, Mazzucato}).

\begin{definition}
Let $1\leq q\leq p<\infty,1\leq r\leq\infty$ and $s\in\mathbb{R}$. The
homogeneous Besov-Morrey space $\mathcal{N}_{p,q,r}^{s}=\mathcal{N}%
_{p,q,r}^{s}(\mathbb{R}^{n})$ is defined as the set off all $u\in
\mathcal{S}_{h}^{^{\prime}}$ such that
\[
\Delta_{j}u\in\mathcal{M}_{q}^{p},\forall j\in\mathbb{Z}%
\]
and
\[
\left\Vert u\right\Vert _{\mathcal{N}_{p,q,r}^{s}}:=\left\Vert \left\{
2^{sj}\parallel\Delta_{j}u\parallel_{\mathcal{M}_{q}^{p}}\right\}
_{j=-\infty}^{\infty}\right\Vert _{\ell^{r}(\mathbb{Z})}<\infty.
\]

\end{definition}

\begin{remark}
The space $\mathcal{N}_{p,q,r}^{s}$, equipped with the norm $\parallel.
\parallel_{\mathcal{N}_{p,q,r}^{s}}$, is a Banach space, and it is
continuously embedded in $\mathcal{S}^{^{\prime}}_{h}$. Moreover, as a direct
consequence of the inclusion relations in $\mathcal{M}^{p}_{q}$ and $\ell
^{r}(\mathbb{Z})$, it follows that $\mathcal{N}_{p,q_{1},r_{1}}^{s}%
\subset\mathcal{N}_{p,q_{2},r_{2}}^{s}$, given that $1\leq q_{2}\leq q_{1}\leq
p<\infty$ and $1\leq r_{1}\leq r_{2}\leq\infty$.
\end{remark}

We will now present two results originally established in
\cite{Kozono-Yamazaki}. The first addresses continuous embeddings involving
Besov and Besov-Morrey spaces and is pivotal for the development of a series
of estimates. Following this, we introduce a Mikhlin-H\"{o}rmander type
theorem applicable to homogeneous Besov-Morrey spaces, which provides a
valuable resource for analyzing the behavior of operators within these spaces.

\begin{lemma}
\label{embeddings} Consider $1\leq q\leq p<\infty, 1\leq r\leq\infty,
s\in\mathbb{R}$ and $\theta\in(0,1)$. Then, we have the continuous inclusions
\[
\mathcal{N}_{p,q,r}^{s}\hookrightarrow\dot{B}_{\infty,r}^{-\frac{n}{p}+s}%
\quad\text{and}\quad\mathcal{N}_{p,q,r}^{s}\hookrightarrow\mathcal{N}%
_{\frac{p}{\theta},\frac{q}{\theta},r}^{-\frac{n}{p}(1-\theta)+s}.
\]

\end{lemma}

\begin{remark}
\label{embedding-in-Linfty} As noted in \cite{Mazzucato}, the above result
also implies the continuous embedding
\[
\mathcal{N}_{p,q,1}^{\frac{n}{p}}\hookrightarrow\mathcal{N}_{p,q,\infty
}^{\frac{n}{p}}\cap L^{\infty}.
\]

\end{remark}

\begin{lemma}
\label{Mikhlin-Hormander} Let $N\geq[n/2]+1$, $1\leq q\leq p<\infty, 1\leq
r\leq\infty$, $m,s\in\mathbb{R}$ and $P$ a operator that, seen as a Fourier
multiplier, has symbol $f$, i.e., $\mathcal{F}[Pu]=f\mathcal{F}[u]$. If $f$ is
a $C^{N}$ function in $\mathbb{R}^{n}\backslash\left\{  0\right\}  $,
satisfying
\[
\vert(\partial^{\alpha}/\partial\xi)f(\xi)\vert\leq M\vert\xi\vert
^{m-\vert\alpha\vert},
\]
for all $\alpha\in\mathbb{N}^{n}$ with $\vert\alpha\vert\leq N$ and some
constant $M>0$, then $P$ is a continuous linear operator from $\mathcal{N}%
_{p,q,r}^{s}$ to $\mathcal{N}_{p,q,r}^{s-m}$ and fulfills the estimate
\[
\Vert P\Vert_{L(\mathcal{N}_{p,q,r}^{s};\mathcal{N}_{p,q,r}^{s-m})}\leq CM,
\]
for some constant $C=C(m,n)>0$.
\end{lemma}

The approach we will use to establish our main results involves employing
fixed-point arguments. In this context, it will be advantageous to
appropriately reformulate the extended stationary Hall-MHD system
(\ref{Ext.HMHD1})-(\ref{Ext.HMHD4}) so that it can be viewed as a generalized
Navier-Stokes system. To clarify our strategy, we will now revisit some
concepts from the theory of well-posedness for the abstract equation
\begin{equation}
\label{abs.eq.}u=\mathcal{L}f + \mathcal{B}(u,u),
\end{equation}
where $f$ is a given data in some Banach space $Y$ and the unknown solution
$u$ of (\ref{abs.eq.}) takes values in a Banach space $X$. Moreover,
$\mathcal{L}:Y \to X$ is a linear operator and $\mathcal{B}:X\times X\to X$ a
bilinear operator, both densely defined. With the aim of simplifying our
analysis and mitigating the need for extensive fixed-point calculations, we
will leverage the following corollary of the Contraction Mapping Theorem in
Banach spaces. But before stating it, let us begin by defining what we mean by
well-posed in the context of equations like (\ref{abs.eq.}). For more details
on this subject and an in-depth discussion, please refer to
\cite{Bejenaru-Tao}.

\begin{definition}
\label{well-posed} The equation (\ref{abs.eq.}) is called well-posed from the
Banach space $Y$ to the Banach space $X$ if there exists constants
$\epsilon,C_{0}>0$ such that

\begin{itemize}
\item[(i)] {For all $f\in B_{Y}(0;\epsilon):=\left\{  f\in Y; \Vert f\Vert
_{Y}<\epsilon\right\}  $, there exists a unique solution $u$ of (\ref{abs.eq.}%
) in $B_{X}(0;C_{0}\epsilon):=\left\{  u\in X;\Vert u\Vert_{X}<C_{0}%
\epsilon\right\}  $.}

\item[(ii)] {The solution map $f\in B_{Y}(0;\epsilon) \mapsto u\in
B_{X}(0;C_{0}\epsilon)$, which is well defined according to (i), is Lipschitz
continuous.}
\end{itemize}
\end{definition}

\begin{lemma}
\label{auxiliar1}(see \cite{Bejenaru-Tao, Lemarie}) Let $X$ and $Y$ Banach
spaces, $\mathcal{L}:Y\rightarrow X$ a linear operator densely defined and
$\mathcal{B}:X\times X\rightarrow X$ a bilinear operator, also densely
defined. If $\mathcal{L}$ and $\mathcal{B}$ are both continuous, i.e., if
there exist constants $C_{1},C_{2}>0$ such that
\[
\Vert\mathcal{L}f\Vert_{X}\leq C_{1}\Vert f\Vert_{Y},\forall f\in Y
\]
and
\[
\Vert\mathcal{B}(u,v)\Vert_{X}\leq C_{2}\Vert u\Vert_{X}\Vert v\Vert
_{X},\forall u,v\in X
\]
then, the equation (\ref{abs.eq.}) is well-posed from $Y$ to $X$, in the sense
of Definition \ref{well-posed}.
\end{lemma}

In order to handle the nonlinear terms in (\ref{Ext.HMHD1})-(\ref{Ext.HMHD4})
and (\ref{aux-NS100}) within the framework of Besov-Morrey spaces, and obtain
estimates suitable for applying Lemma \ref{auxiliar1}, we make use of Bony's
paraproduct decomposition (see \cite{Bony}). For all $u,v\in\mathcal{S}%
_{h}^{^{\prime}}$, this decomposition establishes that%
\begin{equation}
uv=T_{u}v+T_{v}u+R(u,v),\text{ } \label{Bony}%
\end{equation}
where
\[
T_{u}v:=\sum_{j\in\mathbb{Z}}S_{j-2}u\Delta_{j}v,\quad R(u,v):=\sum
_{j\in\mathbb{Z}}\Delta_{j}u\tilde{\Delta}_{j}v\quad\text{and}\quad
\tilde{\Delta}_{j}v:=\sum_{|j-l|\leq1}\Delta_{l}v.
\]

\begin{remark}
\label{Bony-properties}Note that, recalling the identities highlighted in
Remark \ref{operators-properties}, we can promptly establish that
\[
\Delta_{j}(T_{u}v)=\sum_{|j-k|\leq4}\Delta_{j}(S_{k-2}u\Delta_{k}%
v)\quad\text{and}\quad\Delta_{j}(R(u,v))=\sum_{k\geq j-2}\Delta_{j}(\Delta
_{k}u\tilde{\Delta}_{k}v).
\]

\end{remark}

\section{Product estimates}

\label{Sec3} Throughout this section, all content is presented in an
$n$-dimensional setting; that is, we work in $\mathbb{R}^{n}$, with $n\geq1$
assumed unless otherwise specified. Our primary goal is to prove Proposition
\ref{Product:3/p-1} and Proposition \ref{Product:3/p-2}. To achieve this, we
first establish a series of lemmas that provide estimates involving Besov and
Besov-Morrey spaces for the components of Bony's decomposition. These lemmas
will then be used to streamline the arguments in the proofs of the
aforementioned propositions. It is worth noting that some of the estimates
included in the lemmas below can also be found in \cite{Duarte-Ferreira-Villa}%
, albeit with more restrictive assumptions on the regularity index. For other
similar estimates in Besov and Besov-Morrey spaces, see, e.g.,
\cite{Bahouri-Chemin-Danchin, Lemarie, Mazzucato, Perez-Rueda-Villamizar}.

\begin{lemma}
\label{L-inftyxBM} Let $1\leq q\leq p<\infty, 1\leq r\leq\infty$ and
$s\in\mathbb{R}$. Then, $\exists C=C(n,s)>0$ such that
\[
\Vert T_{u}v\Vert_{\mathcal{N}_{p,q,r}^{s}}\leq C\Vert u\Vert_{L^{\infty}%
}\Vert v\Vert_{\mathcal{N}_{p,q,r}^{s}},
\]
for all $u\in L^{\infty}$ and $v\in\mathcal{N}_{p,q,r}^{s}$. Moreover, we
have
\[
\Vert T_{u}v\Vert_{\mathcal{N}_{p,q,r}^{s}}\leq C\Vert u\Vert_{\mathcal{M}%
_{q}^{p}}\Vert v\Vert_{\dot{B}_{\infty,r}^{s}},
\]
for all $u\in\mathcal{M}_{q}^{p}$ and $v\in\dot{B}_{\infty,r}^{s}$.
\end{lemma}

\noindent{\textbf{Proof.}} Initially, note that, by Lemmas
\ref{Delta-j.Sj.Morrey}, \ref{Holder-Morrey} and \ref{Delta-j.Sj.Linfty}, we
have
\begin{align*}
2^{sj}\Vert\Delta_{j}(T_{u}v)\Vert_{\mathcal{M}_{q}^{p}}  &  =2^{sj}\left\Vert
\sum_{|k-j|\leq4}\Delta_{j}(S_{k-2}u\Delta_{k}v)\right\Vert _{\mathcal{M}%
_{q}^{p}}\\
&  \leq C2^{sj}\sum_{|k-j|\leq4}\Vert S_{k-2}u\Delta_{k}v\Vert_{\mathcal{M}%
_{q}^{p}}\\
&  \leq C2^{sj}\sum_{|k-j|\leq4}\Vert S_{k-2}u\Vert_{L^{\infty}}\Vert
\Delta_{k}v\Vert_{\mathcal{M}_{q}^{p}}\\
&  \leq C2^{sj}\Vert u\Vert_{L^{\infty}}\sum_{|k-j|\leq4}\Vert\Delta_{k}%
v\Vert_{\mathcal{M}_{q}^{p}}\\
&  =C\Vert u\Vert_{L^{\infty}}\sum_{|k-j|\leq4}2^{s(j-k)}2^{sk}\Vert\Delta
_{k}v\Vert_{\mathcal{M}_{q}^{p}}\\
&  =C\Vert u\Vert_{L^{\infty}}\left(  \left\{  a_{l}\right\}  \ast\left\{
2^{sl^{^{\prime}}}\Vert\Delta_{l^{^{\prime}}}v\Vert_{\mathcal{M}_{q}^{p}%
}\right\}  \right)  _{j},
\end{align*}
where
\[
a_{l}=\left\{
\begin{array}
[c]{rclll}%
2^{sl}, & \text{if} & \,\,|l|\leq4 &  & \\
0, & \text{if} & \,\,|l|>4 &  &
\end{array}
\right.  .
\]
So, since $\left\{  a_{l}\right\}  \in\ell^{1}(\mathbb{Z})$ and $\left\{
2^{sl^{^{\prime}}}\Vert\Delta_{l^{^{\prime}}}v\Vert_{\mathcal{M}_{q}^{p}%
}\right\}  \in\ell^{r}(\mathbb{Z})$, we have, by Young inequality, that
\begin{align*}
\Vert T_{u}v\Vert_{\mathcal{N}_{p,q,r}^{s}}  &  =\left\Vert \left\{
2^{sj}\Vert\Delta_{j}(T_{u}v)\Vert_{\mathcal{M}_{q}^{p}}\right\}  \right\Vert
_{\ell^{r}(\mathbb{Z})}\\
&  \leq C\Vert u\Vert_{L^{\infty}}\left(  \sum_{|l|\leq4}2^{sl}\right)
\left\Vert \left\{  2^{sj}\Vert\Delta_{j}v\Vert_{\mathcal{M}_{q}^{p}}\right\}
\right\Vert _{\ell^{r}(\mathbb{Z})}\\
&  \leq C\Vert u\Vert_{L^{\infty}}\Vert v\Vert_{\mathcal{N}_{p,q,r}^{s}},
\end{align*}
which proves the first desired inequality. On the other hand, Lemmas
\ref{Delta-j.Sj.Morrey} and \ref{Holder-Morrey} imply that%

\begin{align*}
2^{sj}\Vert\Delta_{j}(T_{u}v)\Vert_{\mathcal{M}_{q}^{p}}  &  =2^{sj}\left\Vert
\sum_{|k-j|\leq4}\Delta_{j}(S_{k-2}u\Delta_{k}v)\right\Vert _{\mathcal{M}%
_{q}^{p}}\\
&  \leq C2^{sj}\sum_{|k-j|\leq4}\Vert S_{k-2}u\Vert_{\mathcal{M}_{q}^{p}}%
\Vert\Delta_{k}v\Vert_{L^{\infty}}\\
&  \leq C2^{sj}\Vert u\Vert_{\mathcal{M}_{q}^{p}}\sum_{|k-j|\leq4}\Vert
\Delta_{k}v\Vert_{L^{\infty}}\\
&  =C\Vert u\Vert_{\mathcal{M}_{q}^{p}}\sum_{|k-j|\leq4}2^{s(j-k)}2^{sk}%
\Vert\Delta_{k}v\Vert_{L^{\infty}}.
\end{align*}
Thus, repeating the previous argument, we conclude the proof.

\begin{flushright}
\ding{110}
\end{flushright}

\begin{lemma}
\label{B-inftyxBM-Holder} Let $1\leq q\leq p<\infty,1\leq r,r_{1},r_{2}%
\leq\infty$ satisfying $\displaystyle\frac{1}{r}=\displaystyle\frac{1}{r_{1}%
}+\displaystyle\frac{1}{r_{2}}$, $s_{1}<0$ and $s_{2}\in\mathbb{R}$. Under
these conditions, there exists a constant $C=C(n,s_{1},s_{2})>0$ such that
\[
\Vert T_{u}v\Vert_{\mathcal{N}_{p,q,r}^{s_{1}+s_{2}}}\leq C\Vert u\Vert
_{\dot{B}_{\infty,r_{1}}^{s_{1}}}\Vert v\Vert_{\mathcal{N}_{p,q,r_{2}}^{s_{2}%
}},
\]
for all $u\in\dot{B}_{\infty,r_{1}}^{s_{1}}$ and $v\in\mathcal{N}_{p,q,r_{2}%
}^{s_{2}}$. Furthermore, the estimate
\[
\Vert T_{u}v\Vert_{\mathcal{N}_{p,q,r}^{s_{1}+s_{2}}}\leq C\Vert
u\Vert_{\mathcal{N}_{p,q,r_{1}}^{s_{1}}}\Vert v\Vert_{\dot{B}_{\infty,r_{2}%
}^{s_{2}}}%
\]
holds true, for all $u\in\mathcal{N}_{p,q,r_{1}}^{s_{1}}$ and $v\in\dot
{B}_{\infty,r_{2}}^{s_{2}}$.
\end{lemma}

\noindent{\textbf{Proof.}} Making use of Lemmas \ref{Delta-j.Sj.Morrey} and
\ref{Holder-Morrey}, we get
\begin{align}
2^{(s_{1}+s_{2})j}\Vert\Delta_{j}(T_{u}v)\Vert_{\mathcal{M}_{q}^{p}}  &
=2^{(s_{1}+s_{2})j}\left\Vert \sum_{|k-j|\leq4}\Delta_{j}(S_{k-2}u\Delta
_{k}v)\right\Vert _{\mathcal{M}_{q}^{p}}\nonumber\\
&  \leq C2^{(s_{1}+s_{2})j}\sum_{|k-j|\leq4}\left\Vert S_{k-2}u\Delta
_{k}v\right\Vert _{\mathcal{M}_{q}^{p}}\nonumber\\
&  \leq C2^{(s_{1}+s_{2})j}\sum_{|k-j|\leq4}\Vert S_{k-2}u\Vert_{L^{\infty}%
}\Vert\Delta_{k}v\Vert_{\mathcal{M}_{q}^{p}}\nonumber\\
&  =C\sum_{|k-j|\leq4}2^{(s_{1}+s_{2})(j-k)}2^{s_{1}k}\Vert S_{k-2}%
u\Vert_{L^{\infty}}2^{s_{2}k}\Vert\Delta_{k}v\Vert_{\mathcal{M}_{q}^{p}%
}\nonumber\\
&  =C\left(  \left\{  a_{l^{^{\prime}}}\right\}  \ast\left\{  2^{s_{1}l}\Vert
S_{l-2}u\Vert_{L^{\infty}}2^{s_{2}l}\Vert\Delta_{l}v\Vert_{\mathcal{M}_{q}%
^{p}}\right\}  \right)  _{j}, \label{Young1}%
\end{align}
where
\[
a_{l^{^{\prime}}}=\left\{
\begin{array}
[c]{rclll}%
2^{(s_{1}+s_{2})l^{^{\prime}}}, & \text{if} & \,\,|l^{^{\prime}}|\leq4 &  & \\
0, & \text{if} & \,\,|l^{^{\prime}}|>4 &  &
\end{array}
\right.  .
\]
Note now that
\begin{align*}
2^{s_{1}l}\Vert S_{l-2}u\Vert_{L^{\infty}}  &  =2^{s_{1}l}\left\Vert
\sum_{m\leq l-2}\Delta_{m}u\right\Vert _{L^{\infty}}\\
&  \leq2^{s_{1}l}\sum_{m\leq l-2}\Vert\Delta_{m}u\Vert_{L^{\infty}}\\
&  =\sum_{m\leq l-2}2^{s_{1}(l-m)}2^{s_{1}m}\Vert\Delta_{m}u\Vert_{L^{\infty}%
}\\
&  =\left(  \left\{  b_{l^{^{\prime}}}\right\}  \ast\left\{  2^{s_{1}%
k^{^{\prime}}}\Vert\Delta_{k^{^{\prime}}}u\Vert_{L^{\infty}}\right\}  \right)
_{l},
\end{align*}
where
\[
b_{l^{^{\prime}}}=\left\{
\begin{array}
[c]{rclll}%
2^{s_{1}l^{^{\prime}}}, & \text{if} & \,\,l^{^{\prime}}\geq2 &  & \\
0, & \text{if} & \,\,l^{^{\prime}}<2 &  &
\end{array}
\right.  .
\]
So, since $\left\{  b_{l^{^{\prime}}}\right\}  \in\ell^{1}(\mathbb{Z})$,
because $s_{1}<0$, and $\left\{  2^{s_{1}k^{^{\prime}}}\Vert\Delta
_{k^{^{\prime}}}u\Vert_{L^{\infty}}\right\}  \in\ell^{r_{1}}(\mathbb{Z})$, we
have, as a consequence of Young inequality, that $\left\{  2^{s_{1}l}\Vert
S_{l-2}u\Vert_{L^{\infty}}\right\}  \in\ell^{r_{1}}(\mathbb{Z})$ and
\[
\left\Vert \left\{  2^{s_{1}l}\Vert S_{l-2}u\Vert_{L^{\infty}}\right\}
\right\Vert _{\ell^{r_{1}}(\mathbb{Z})}\leq\left(  \sum_{l^{^{\prime}}\geq
2}2^{s_{1}l^{^{\prime}}}\right)  \left\Vert \left\{  2^{s_{1}l}\Vert\Delta
_{l}u\Vert_{L^{\infty}}\right\}  \right\Vert _{\ell^{r_{1}}(\mathbb{Z})}\leq
C\Vert u\Vert_{\dot{B}_{\infty,r_{1}}^{s_{1}}}.
\]
Now, in view of the fact that $\left\{  2^{s_{2}l}\Vert\Delta_{l}%
v\Vert_{\mathcal{M}_{q}^{p}}\right\}  \in\ell^{r_{2}}(\mathbb{Z})$, we
immediately obtain, by H\"{o}lder inequality in sequence spaces, that
$\left\{  2^{s_{1}l}\Vert S_{l-2}u\Vert_{L^{\infty}}2^{s_{2}l}\Vert\Delta
_{l}v\Vert_{\mathcal{M}_{q}^{p}}\right\}  \in\ell^{r}(\mathbb{Z})$ and
\begin{align}
&  \left\Vert \left\{  2^{s_{1}l}\Vert S_{l-2}u\Vert_{L^{\infty}}2^{s_{2}%
l}\Vert\Delta_{l}v\Vert_{\mathcal{M}_{q}^{p}}\right\}  \right\Vert _{\ell
^{r}(\mathbb{Z})}\nonumber\\
&  \leq\left\Vert \left\{  2^{s_{1}l}\Vert S_{l-2}u\Vert_{L^{\infty}}\right\}
\right\Vert _{\ell^{r_{1}}(\mathbb{Z})}\left\Vert \left\{  2^{s_{2}l}%
\Vert\Delta_{l}v\Vert_{\mathcal{M}_{q}^{p}}\right\}  \right\Vert _{\ell
^{r_{2}}(\mathbb{Z})}\nonumber\\
&  \leq C\Vert u\Vert_{\dot{B}_{\infty,r_{1}}^{s_{1}}}\Vert v\Vert
_{\mathcal{N}_{p,q,r_{2}}^{s_{2}}}. \label{Holder1}%
\end{align}
Therefore , since $\left\{  a_{l^{^{\prime}}}\right\}  \in\ell^{1}%
(\mathbb{Z})$, it follows from Young inequality and estimates (\ref{Young1})
and (\ref{Holder1}) that
\begin{align*}
\Vert T_{u}v\Vert_{\mathcal{N}_{p,q,r}^{s_{1}+s_{2}}}  &  =\left\Vert \left\{
2^{(s_{1}+s_{2})j}\Vert\Delta_{j}(T_{u}v)\Vert_{\mathcal{M}_{q}^{p}}\right\}
\right\Vert _{\ell^{r}(\mathbb{Z})}\\
&  \leq\left(  \sum_{|l^{^{\prime}}|\leq4}2^{(s_{1}+s_{2})l^{^{\prime}}%
}\right)  \left\Vert \left\{  2^{s_{1}l}\Vert S_{l-2}u\Vert_{L^{\infty}%
}2^{s_{2}l}\Vert\Delta_{l}v\Vert_{\mathcal{M}_{q}^{p}}\right\}  \right\Vert
_{\ell^{r}(\mathbb{Z})}\\
&  \leq C\Vert u\Vert_{\dot{B}_{\infty,r_{1}}^{s_{1}}}\Vert v\Vert
_{\mathcal{N}_{p,q,r_{2}}^{s_{2}}},
\end{align*}
which proves the first estimate. Alternatively, Lemmas \ref{Delta-j.Sj.Morrey}
and \ref{Holder-Morrey} give us that%

\begin{align}
2^{(s_{1}+s_{2})j}\Vert\Delta_{j}(T_{u}v)\Vert_{\mathcal{M}_{q}^{p}}  &
=2^{(s_{1}+s_{2})j}\left\Vert \sum_{\vert k-j\vert\leq4}\Delta_{j}%
(S_{k-2}u\Delta_{k}v)\right\Vert _{\mathcal{M}_{q}^{p}}\nonumber\\
&  \leq C2^{(s_{1}+s_{2})j}\sum_{\vert k-j\vert\leq4}\left\Vert S_{k-2}%
u\Delta_{k}v \right\Vert _{\mathcal{M}_{q}^{p}}\nonumber\\
&  \leq C2^{(s_{1}+s_{2})j}\sum_{\vert k-j\vert\leq4}\Vert S_{k-2}%
u\Vert_{\mathcal{M}_{q}^{p}}\Vert\Delta_{k}v\Vert_{L^{\infty}}\nonumber\\
&  = C\sum_{\vert k-j\vert\leq4}2^{(s_{1}+s_{2})(j-k)}2^{s_{1}k}\Vert
S_{k-2}u\Vert_{\mathcal{M}_{q}^{p}}2^{s_{2}k}\Vert\Delta_{k}v\Vert_{L^{\infty
}}\nonumber\\
&  =C\left(  \left\{  a_{l^{^{\prime}}}\right\}  *\left\{  2^{s_{1}l}\Vert
S_{l-2}u\Vert_{\mathcal{M}_{q}^{p}}2^{s_{2}l}\Vert\Delta_{l}v\Vert_{L^{\infty
}}\right\}  \right)  _{j}. \label{Young1-2}%
\end{align}
Thus, an analogous argument to the one made above leads us to conclude that
\[
\left\Vert \left\{  2^{s_{1}l}\Vert S_{l-2}u\Vert_{\mathcal{M}_{q}^{p}%
}\right\}  \right\Vert _{\ell^{r_{1}}(\mathbb{Z})}\leq C\Vert u\Vert
_{\mathcal{N}_{p,q,r_{1}}^{s_{1}}}
\]
and%

\begin{equation}
\left\Vert \left\{  2^{s_{1}l}\Vert S_{l-2}u\Vert_{\mathcal{M}_{q}^{p}%
}2^{s_{2}l}\Vert\Delta_{l}v\Vert_{L^{\infty}}\right\}  \right\Vert _{\ell
^{r}(\mathbb{Z})}\leq C\Vert u\Vert_{\mathcal{N}_{p,q,r_{1}}^{s_{1}}}\Vert
v\Vert_{\dot{B}_{\infty,r_{2}}^{s_{2}}}. \label{Holder1-2}%
\end{equation}
Applying Young inequality to (\ref{Young1-2}) and using estimate
(\ref{Holder1-2}), we then obtain%

\begin{align*}
\Vert T_{u}v\Vert_{\mathcal{N}_{p,q,r}^{s_{1}+s_{2}}}  &  =\left\Vert \left\{
2^{(s_{1}+s_{2})j}\Vert\Delta_{j}(T_{u}v)\Vert_{\mathcal{M}_{q}^{p}}\right\}
\right\Vert _{\ell^{r}(\mathbb{Z})}\\
&  \leq\left(  \sum_{\vert l^{^{\prime}}\vert\leq4}2^{(s_{1}+s_{2}%
)l^{^{\prime}}}\right)  \left\Vert \left\{  2^{s_{1}l}\Vert S_{l-2}%
u\Vert_{\mathcal{M}_{q}^{p}}2^{s_{2}l}\Vert\Delta_{l}v\Vert_{L^{\infty}%
}\right\}  \right\Vert _{\ell^{r}(\mathbb{Z})}\\
&  \leq C\Vert u\Vert_{\mathcal{N}_{p,q,r_{1}}^{s_{1}}}\Vert v\Vert_{\dot
{B}_{\infty,r_{2}}^{s_{2}}},
\end{align*}
and the proof is complete.

\begin{flushright}
\ding{110}
\end{flushright}

\begin{lemma}
\label{R-Binfity-BM-Holder} Let $1\leq q\leq p<\infty,1\leq r,r_{1},r_{2}%
\leq\infty$ with $\displaystyle\frac{1}{r}=\displaystyle\frac{1}{r_{1}%
}+\displaystyle\frac{1}{r_{2}}$ and $s_{1},s_{2}\in\mathbb{R}$ satisfying
$s_{1}+s_{2}>0$. Then, $\exists C=C(n,s_{1},s_{2})>0$ such that
\[
\Vert R(u,v)\Vert_{\mathcal{N}_{p,q,r}^{s_{1}+s_{2}}}\leq C\Vert u\Vert
_{\dot{B}_{\infty,r_{1}}^{s_{1}}}\Vert v\Vert_{\mathcal{N}_{p,q,r_{2}}^{s_{2}%
}},
\]
for all $u\in\dot{B}_{\infty,r_{1}}^{s_{1}}$ and $v\in\mathcal{N}_{p,q,r_{2}%
}^{s_{2}}$. Additionally, the estimate
\[
\Vert R(u,v)\Vert_{\mathcal{N}_{p,q,r}^{s_{1}+s_{2}}}\leq C\Vert
u\Vert_{\mathcal{N}_{p,q,r_{1}}^{s_{1}}}\Vert v\Vert_{\dot{B}_{\infty,r_{2}%
}^{s_{2}}}%
\]
holds true, for all $u\in\mathcal{N}_{p,q,r_{1}}^{s_{1}}$ and $v\in\dot
{B}_{\infty,r_{2}}^{s_{2}}$.
\end{lemma}

\noindent{\textbf{Proof.}} Note that, by Lemmas \ref{Delta-j.Sj.Morrey} and
\ref{Holder-Morrey}, we have
\begin{align}
2^{\left(  s_{1}+s_{2}\right)  j}\Vert\Delta_{j}(R(u,v))\Vert_{\mathcal{M}%
_{q}^{p}}  &  =2^{\left(  s_{1}+s_{2}\right)  j}\left\Vert \sum_{k\geq
j-2}\Delta_{j}(\Delta_{k}u\tilde{\Delta}_{k}v)\right\Vert _{\mathcal{M}%
_{q}^{p}}\nonumber\\
&  \leq C2^{\left(  s_{1}+s_{2}\right)  j}\sum_{k\geq j-2}\Vert\Delta
_{k}u\tilde{\Delta}_{k}v\Vert_{\mathcal{M}_{q}^{p}}\nonumber\\
&  \leq C2^{\left(  s_{1}+s_{2}\right)  j}\sum_{k\geq j-2}\Vert\Delta
_{k}u\Vert_{L^{\infty}}\Vert\tilde{\Delta}_{k}v\Vert_{\mathcal{M}_{q}^{p}%
}\nonumber\\
&  =C\sum_{k\geq j-2}2^{(s_{1}+s_{2})(j-k)}2^{s_{1}k}\Vert\Delta_{k}%
u\Vert_{L^{\infty}}2^{s_{2}k}\Vert\tilde{\Delta}_{k}v\Vert_{\mathcal{M}%
_{q}^{p}}\nonumber\\
&  =C\left(  \left\{  a_{l^{^{\prime}}}\right\}  \ast\left\{  2^{s_{1}l}%
\Vert\Delta_{l}u\Vert_{L^{\infty}}2^{s_{2}l}\Vert\tilde{\Delta}_{l}%
v\Vert_{\mathcal{M}_{q}^{p}}\right\}  \right)  _{j}, \label{Young2}%
\end{align}
where
\[
a_{l^{^{\prime}}}=\left\{
\begin{array}
[c]{rclll}%
2^{(s_{1}+s_{2})l^{^{\prime}}}, & \text{if} & \,\,l^{^{\prime}}\leq2 &  & \\
0, & \text{if} & \,\,l^{^{\prime}}>2 &  &
\end{array}
\right.  .
\]
Furthermore, we have
\begin{align*}
2^{s_{2}l}\Vert\tilde{\Delta}_{l}v\Vert_{\mathcal{M}_{q}^{p}}  &  =2^{s_{2}%
l}\left\Vert \sum_{|m-l|\leq1}\Delta_{m}v\right\Vert _{\mathcal{M}_{q}^{p}}\\
&  \leq2^{s_{2}l}\sum_{|m-l|\leq1}\Vert\Delta_{m}v\Vert_{\mathcal{M}_{q}^{p}%
}\\
&  =\sum_{|m-l|\leq1}2^{s_{2}(l-m)}2^{s_{2}m}\Vert\Delta_{m}v\Vert
_{\mathcal{M}_{q}^{p}}\\
&  =\left(  \left\{  b_{l^{^{\prime}}}\right\}  \ast\left\{  2^{s_{2}%
k^{^{\prime}}}\Vert\Delta_{k^{^{\prime}}}v\Vert_{\mathcal{M}_{q}^{p}}\right\}
\right)  _{l},
\end{align*}
where
\[
b_{l^{^{\prime}}}=\left\{
\begin{array}
[c]{rclll}%
2^{s_{2}l^{^{\prime}}}, & \text{if} & \,\,|l^{^{\prime}}|\leq1 &  & \\
0, & \text{if} & \,\,|l^{^{\prime}}|>1 &  &
\end{array}
\right.  .
\]
So, making use of Young inequality, we have
\[
\left\Vert \left\{  2^{s_{2}l}\Vert\tilde{\Delta}_{l}v\Vert_{\mathcal{M}%
_{q}^{p}}\right\}  \right\Vert _{\ell^{r_{2}}(\mathbb{Z})}\leq\left(
\sum_{|l^{^{\prime}}|\leq1}2^{s_{2}l^{^{\prime}}}\right)  \left\Vert \left\{
2^{s_{2}l}\Vert\Delta_{l}v\Vert_{\mathcal{M}_{q}^{p}}\right\}  \right\Vert
_{\ell^{r_{2}}(\mathbb{Z})}\leq C\Vert v\Vert_{\mathcal{N}_{p,q,r_{2}}^{s_{2}%
}}.
\]
Consequently, through H\"{o}lder inequality in sequence spaces, we get the
estimate
\begin{align}
&  \left\Vert \left\{  2^{s_{1}l}\Vert\Delta_{l}u\Vert_{L^{\infty}}2^{s_{2}%
l}\Vert\tilde{\Delta}_{l}v\Vert_{\mathcal{M}_{q}^{p}}\right\}  \right\Vert
_{\ell^{r}(\mathbb{Z})}\nonumber\\
&  \leq\left\Vert \left\{  2^{s_{1}l}\Vert\Delta_{l}u\Vert_{L^{\infty}%
}\right\}  \right\Vert _{\ell^{r_{1}}(\mathbb{Z})}\left\Vert \left\{
2^{s_{2}l}\Vert\tilde{\Delta}_{l}v\Vert_{\mathcal{M}_{q}^{p}}\right\}
\right\Vert _{\ell^{r_{2}}(\mathbb{Z})}\nonumber\\
&  \leq C\Vert u\Vert_{\dot{B}_{\infty,r_{1}}^{s_{1}}}\Vert v\Vert
_{\mathcal{N}_{p,q,r_{2}}^{s_{2}}}. \label{Holder2}%
\end{align}
Thus, since $\left\{  a_{l^{^{\prime}}}\right\}  \in\ell^{1}(\mathbb{Z})$,
because we have the hypothesis $s_{1}+s_{2}>0$, we can return to estimate
(\ref{Young2}) and get, as a consequence of Young inequality and estimate
(\ref{Holder2}), that
\begin{align*}
\Vert R(u,v)\Vert_{\mathcal{N}_{p,q,r}^{s_{1}+s_{2}}}  &  =\left\Vert \left\{
2^{(s_{1}+s_{2})j}\Vert\Delta_{j}(R(u,v))\Vert_{\mathcal{M}_{q}^{p}}\right\}
\right\Vert _{\ell^{r}(\mathbb{Z})}\\
&  \leq\left(  \sum_{l^{^{\prime}}\leq2}2^{(s_{1}+s_{2})l^{^{\prime}}}\right)
\left\Vert \left\{  2^{s_{1}l}\Vert\Delta_{l}u\Vert_{L^{\infty}}2^{s_{2}%
l}\Vert\tilde{\Delta}_{l}v\Vert_{\mathcal{M}_{q}^{p}}\right\}  \right\Vert
_{\ell^{r}(\mathbb{Z})}\\
&  \leq C\Vert u\Vert_{\dot{B}_{\infty,r_{1}}^{s_{1}}}\Vert v\Vert
_{\mathcal{N}_{p,q,r_{2}}^{s_{2}}},
\end{align*}
which proves the first part of the result. Now, using Lemmas
\ref{Delta-j.Sj.Morrey} and \ref{Holder-Morrey} again, we can also obtain that%

\begin{align}
2^{\left(  s_{1}+s_{2}\right)  j}\Vert\Delta_{j}(R(u,v))\Vert_{\mathcal{M}%
_{q}^{p}}  &  =2^{\left(  s_{1}+s_{2}\right)  j}\left\Vert \sum_{k\geq
j-2}\Delta_{j}(\Delta_{k}u\tilde{\Delta}_{k}v)\right\Vert _{\mathcal{M}%
_{q}^{p}}\nonumber\\
&  \leq C2^{\left(  s_{1}+s_{2}\right)  j}\sum_{k\geq j-2}\Vert\Delta
_{k}u\tilde{\Delta}_{k}v\Vert_{\mathcal{M}_{q}^{p}}\nonumber\\
&  \leq C2^{\left(  s_{1}+s_{2}\right)  j}\sum_{k\geq j-2}\Vert\Delta
_{k}u\Vert_{\mathcal{M}_{q}^{p}}\Vert\tilde{\Delta}_{k}v\Vert_{L^{\infty}%
}\nonumber\\
&  = C\sum_{k\geq j-2}2^{(s_{1}+s_{2})(j-k)}2^{s_{1}k}\Vert\Delta_{k}%
u\Vert_{\mathcal{M}_{q}^{p}}2^{s_{2}k}\Vert\tilde{\Delta}_{k}v\Vert
_{L^{\infty}}\nonumber\\
&  =C\left(  \left\{  a_{l^{^{\prime}}}\right\}  *\left\{  2^{s_{1}l}%
\Vert\Delta_{l}u\Vert_{\mathcal{M}_{q}^{p}}2^{s_{2}l}\Vert\tilde{\Delta}%
_{l}v\Vert_{L^{\infty}}\right\}  \right)  _{j}. \label{Young2-2}%
\end{align}
And then, proceeding as before, we have
\[
\left\Vert \left\{  2^{s_{2}l}\Vert\tilde{\Delta}_{l}v\Vert_{L^{\infty}%
}\right\}  \right\Vert _{\ell^{r_{2}}(\mathbb{Z})}\leq C\Vert v\Vert_{\dot
{B}_{\infty,r_{2}}^{s_{2}}}
\]
and%

\begin{equation}
\left\Vert \left\{  2^{s_{1}l}\Vert\Delta_{l}u\Vert_{\mathcal{M}_{q}^{p}%
}2^{s_{2}l}\Vert\tilde{\Delta}_{l}v\Vert_{L^{\infty}}\right\}  \right\Vert
_{\ell^{r}(\mathbb{Z})}\leq C\Vert u\Vert_{\mathcal{N}_{p,q,r_{1}}^{s_{1}}%
}\Vert v\Vert_{\dot{B}_{\infty,r_{2}}^{s_{2}}}. \label{Holder2-2}%
\end{equation}
In that case, the combination of Young inequality and estimates
(\ref{Young2-2}) and (\ref{Holder2-2}) gives us that%

\begin{align*}
\Vert R(u,v)\Vert_{\mathcal{N}_{p,q,r}^{s_{1}+s_{2}}}  &  =\left\Vert \left\{
2^{(s_{1}+s_{2})j}\Vert\Delta_{j}(R(u,v))\Vert_{\mathcal{M}_{q}^{p}}\right\}
\right\Vert _{\ell^{r}(\mathbb{Z})}\\
&  \leq\left(  \sum_{l^{^{\prime}}\leq2}2^{(s_{1}+s_{2})l^{^{\prime}}}\right)
\left\Vert \left\{  2^{s_{1}l}\Vert\Delta_{l}u\Vert_{\mathcal{M}_{q}^{p}%
}2^{s_{2}l}\Vert\tilde{\Delta}_{l}v\Vert_{L^{\infty}}\right\}  \right\Vert
_{\ell^{r}(\mathbb{Z})}\\
&  \leq C\Vert u\Vert_{\mathcal{N}_{p,q,r_{1}}^{s_{1}}}\Vert v\Vert_{\dot
{B}_{\infty,r_{2}}^{s_{2}}},
\end{align*}
thereby concluding the result.

\begin{flushright}
\ding{110}
\end{flushright}

\begin{lemma}
\label{R-BM-BM-Holder} Let $1\leq q\leq p<\infty$ and $1\leq q_{i}\leq
p_{i}<\infty$, $i=1,2$, such that $\displaystyle\frac{1}{p}=\displaystyle\frac
{1}{p_{1}}+\displaystyle\frac{1}{p_{2}}$ and $\displaystyle\frac{1}{q}%
\geq\displaystyle\frac{1}{q_{1}}+\displaystyle\frac{1}{q_{2}}$. Furthermore,
let $1\leq r,r_{1},r_{2}\leq\infty$ with $\displaystyle\frac{1}{r}%
=\displaystyle\frac{1}{r_{1}}+\displaystyle\frac{1}{r_{2}}$ and $s_{1}%
,s_{2}\in\mathbb{R}$ satisfying $s_{1}+s_{2}>0$. Then, there exists a positive
constant $C=C(n,s_{1}, s_{2})$ such that
\[
\Vert R(u,v)\Vert_{\mathcal{N}_{p,q,r}^{s_{1}+s_{2}}}\leq C\Vert
u\Vert_{\mathcal{N}_{p_{1},q_{1},r_{1}}^{s_{1}}}\Vert v\Vert_{\mathcal{N}%
_{p_{2},q_{2},r_{2}}^{s_{2}}},
\]
for all $u\in\mathcal{N}_{p_{1},q_{1},r_{1}}^{s_{1}}$ and $v\in\mathcal{N}%
_{p_{2},q_{2},r_{2}}^{s_{2}}$.
\end{lemma}

\noindent{\textbf{Proof.}} Indeed, applying Lemmas \ref{Delta-j.Sj.Morrey} and
\ref{Holder-Morrey}, we have that
\begin{align}
2^{\left(  s_{1}+s_{2}\right)  j}\Vert\Delta_{j}(R(u,v))\Vert_{\mathcal{M}%
_{q}^{p}}  &  =2^{\left(  s_{1}+s_{2}\right)  j}\left\Vert \sum_{k\geq
j-2}\Delta_{j}(\Delta_{k}u\tilde{\Delta}_{k}v)\right\Vert _{\mathcal{M}%
_{q}^{p}}\nonumber\\
&  \leq C2^{\left(  s_{1}+s_{2}\right)  j}\sum_{k\geq j-2}\Vert\Delta
_{k}u\tilde{\Delta}_{k}v\Vert_{\mathcal{M}_{q}^{p}}\nonumber\\
&  \leq C2^{\left(  s_{1}+s_{2}\right)  j}\sum_{k\geq j-2}\Vert\Delta
_{k}u\Vert_{\mathcal{M}_{q_{1}}^{p_{1}}}\Vert\tilde{\Delta}_{k}v\Vert
_{\mathcal{M}_{q_{2}}^{p_{2}}}\nonumber\\
&  =C\sum_{k\geq j-2}2^{(s_{1}+s_{2})(j-k)}2^{s_{1}k}\Vert\Delta_{k}%
u\Vert_{\mathcal{M}_{q_{1}}^{p_{1}}}2^{s_{2}k}\Vert\tilde{\Delta}_{k}%
v\Vert_{\mathcal{M}_{q_{2}}^{p_{2}}}\nonumber\\
&  =C\left(  \left\{  a_{l^{^{\prime}}}\right\}  \ast\left\{  2^{s_{1}l}%
\Vert\Delta_{l}u\Vert_{\mathcal{M}_{q_{1}}^{p_{1}}}2^{s_{2}l}\Vert
\tilde{\Delta}_{l}v\Vert_{\mathcal{M}_{q_{2}}^{p_{2}}}\right\}  \right)  _{j},
\label{Young3}%
\end{align}
where
\[
a_{l^{^{\prime}}}=\left\{
\begin{array}
[c]{rclll}%
2^{(s_{1}+s_{2})l^{^{\prime}}}, & \text{if} & \,\,l^{^{\prime}}\leq2 &  & \\
0, & \text{if} & \,\,l^{^{\prime}}>2 &  &
\end{array}
\right.  .
\]
Now, in the same way as in Lemma \ref{R-Binfity-BM-Holder}, we have that
\[
\left\Vert \left\{  2^{s_{2}l}\Vert\tilde{\Delta}_{l}v\Vert_{\mathcal{M}%
_{q_{2}}^{p_{2}}}\right\}  \right\Vert _{\ell^{r_{2}}(\mathbb{Z})}\leq C\Vert
v\Vert_{\mathcal{N}_{p_{2},q_{2},r_{2}}^{s_{2}}}.
\]
Then, H\"{o}lder inequality in sequence spaces provides us with the estimate
\begin{align}
&  \left\Vert \left\{  2^{s_{1}l}\Vert\Delta_{l}u\Vert_{\mathcal{M}_{q_{1}%
}^{p_{1}}}2^{s_{2}l}\Vert\tilde{\Delta}_{l}v\Vert_{\mathcal{M}_{q_{2}}^{p_{2}%
}}\right\}  \right\Vert _{\ell^{r}(\mathbb{Z})}\nonumber\\
&  \leq\left\Vert \left\{  2^{s_{1}l}\Vert\Delta_{l}u\Vert_{\mathcal{M}%
_{q_{1}}^{p_{1}}}\right\}  \right\Vert _{\ell^{r_{1}}(\mathbb{Z})}\left\Vert
\left\{  2^{s_{2}l}\Vert\tilde{\Delta}_{l}v\Vert_{\mathcal{M}_{q_{2}}^{p_{2}}%
}\right\}  \right\Vert _{\ell^{r_{2}}(\mathbb{Z})}\nonumber\\
&  \leq C\Vert u\Vert_{\mathcal{N}_{p_{1},q_{1},r_{1}}^{s_{1}}}\Vert
v\Vert_{\mathcal{N}_{p_{2},q_{2},r_{2}}^{s_{2}}}. \label{Holder3}%
\end{align}
Thus, since the hypothesis $s_{1}+s_{2}>0$ implies $\left\{  a_{l^{^{\prime}}%
}\right\}  \in\ell^{1}(\mathbb{Z})$, we can apply Young inequality to estimate
(\ref{Young3}) and obtain, with the aid of estimate (\ref{Holder3}), that
\begin{align*}
\Vert R(u,v)\Vert_{\mathcal{N}_{p,q,r}^{s_{1}+s_{2}}}  &  =\left\Vert \left\{
2^{(s_{1}+s_{2})j}\Vert\Delta_{j}(R(u,v))\Vert_{\mathcal{M}_{q}^{p}}\right\}
\right\Vert _{\ell^{r}(\mathbb{Z})}\\
&  \leq\left(  \sum_{l^{^{\prime}}\leq2}2^{(s_{1}+s_{2})l^{^{\prime}}}\right)
\left\Vert \left\{  2^{s_{1}l}\Vert\Delta_{l}u\Vert_{\mathcal{M}_{q_{1}%
}^{p_{1}}}2^{s_{2}l}\Vert\tilde{\Delta}_{l}v\Vert_{\mathcal{M}_{q_{2}}^{p_{2}%
}}\right\}  \right\Vert _{\ell^{r}(\mathbb{Z})}\\
&  \leq C\Vert u\Vert_{\mathcal{N}_{p_{1},q_{1},r_{1}}^{s_{1}}}\Vert
v\Vert_{\mathcal{N}_{p_{2},q_{2},r_{2}}^{s_{2}}},
\end{align*}
from which the result follows.

\begin{flushright}
\ding{110}
\end{flushright}

Having these results in hand, we are now ready to prove Proposition
\ref{Product:3/p-1} and Proposition \ref{Product:3/p-2}.

\noindent{\textbf{Proof of Proposition \ref{Product:3/p-1}.}} Indeed, Lemma
\ref{B-inftyxBM-Holder} and the continuous embedding $\mathcal{N}%
_{p,q,r}^{\frac{n}{p}-1}\hookrightarrow\dot{B}_{\infty,r}^{-1}$ (Lemma
\ref{embeddings}) imply that
\begin{equation}
\Vert T_{u}v\Vert_{\mathcal{N}_{p,q,r}^{\frac{n}{p}-1}}\leq C\Vert
u\Vert_{\dot{B}_{\infty,r}^{-1}}\Vert v\Vert_{\mathcal{N}_{p,q,\infty}%
^{\frac{n}{p}}}\leq C\Vert u\Vert_{\mathcal{N}_{p,q,r}^{\frac{n}{p}-1}}\Vert
v\Vert_{\mathcal{N}_{p,q,\infty}^{\frac{n}{p}}}. \label{Prod.1}%
\end{equation}
Also, Lemma \ref{L-inftyxBM} gives us that
\begin{equation}
\Vert T_{v}u\Vert_{\mathcal{N}_{p,q,r}^{\frac{n}{p}-1}}\leq C\Vert
v\Vert_{L^{\infty}}\Vert u\Vert_{\mathcal{N}_{p,q,r}^{\frac{n}{p}-1}}.
\label{Prod.2}%
\end{equation}
And by Lemma \ref{R-Binfity-BM-Holder} and the continuous embedding
$\mathcal{N}_{p,q,\infty}^{\frac{n}{p}}\hookrightarrow\dot{B}_{\infty,\infty
}^{0}$ (Lemma \ref{embeddings}), we have
\begin{equation}
\Vert R(u,v)\Vert_{\mathcal{N}_{p,q,r}^{\frac{n}{p}-1}}\leq C\Vert
u\Vert_{\mathcal{N}_{p,q,r}^{\frac{n}{p}-1}}\Vert v\Vert_{\dot{B}%
_{\infty,\infty}^{0}}\leq C\Vert u\Vert_{\mathcal{N}_{p,q,r}^{\frac{n}{p}-1}%
}\Vert v\Vert_{\mathcal{N}_{p,q,\infty}^{\frac{n}{p}}}. \label{Prod.3}%
\end{equation}
Therefore, the first estimate in the statement follows from Bony's
decomposition (\ref{Bony}) and estimates (\ref{Prod.1}), (\ref{Prod.2}) and
(\ref{Prod.3}). In turn, the second one is an immediate consequence of the
first estimate and the continuous embedding $\mathcal{N}_{p,q,1}^{\frac{n}{p}%
}\hookrightarrow\mathcal{N}_{p,q,\infty}^{\frac{n}{p}}\cap L^{\infty}$ (see
Remark \ref{embedding-in-Linfty}).

\begin{flushright}
\ding{110}
\end{flushright}

\noindent{\textbf{Proof of Proposition \ref{Product:3/p-2}.}} First, note that
Lemma \ref{B-inftyxBM-Holder} together with the continuous embeddings
$\mathcal{N}_{p,q,r_{1}}^{\frac{n}{p}-1}\hookrightarrow\dot{B}_{\infty,r_{1}%
}^{-1}$ and $\mathcal{N}_{p,q,r_{2}}^{\frac{n}{p}-1}\hookrightarrow\dot
{B}_{\infty,r_{2}}^{-1}$ (see Lemma \ref{embeddings}) provides us with the
estimates
\begin{equation}
\Vert T_{u}v\Vert_{\mathcal{N}_{p,q,r}^{\frac{n}{p}-2}}\leq C\Vert
u\Vert_{\dot{B}_{\infty,r_{1}}^{-1}}\Vert v\Vert_{\mathcal{N}_{p,q,r_{2}%
}^{\frac{n}{p}-1}}\leq C\Vert u\Vert_{\mathcal{N}_{p,q,r_{1}}^{\frac{n}{p}-1}%
}\Vert v\Vert_{\mathcal{N}_{p,q,r_{2}}^{\frac{n}{p}-1}} \label{Prod.4}%
\end{equation}
and
\begin{equation}
\Vert T_{v}u\Vert_{\mathcal{N}_{p,q,r}^{\frac{n}{p}-2}}\leq C\Vert
v\Vert_{\dot{B}_{\infty,r_{2}}^{-1}}\Vert u\Vert_{\mathcal{N}_{p,q,r_{1}%
}^{\frac{n}{p}-1}}\leq C\Vert v\Vert_{\mathcal{N}_{p,q,r_{2}}^{\frac{n}{p}-1}%
}\Vert u\Vert_{\mathcal{N}_{p,q,r_{1}}^{\frac{n}{p}-1}}. \label{Prod.5}%
\end{equation}
Now, if $(i)$ holds, we can apply Lemma \ref{R-Binfity-BM-Holder} and use the
same embedding above to get the estimate%

\begin{equation}
\Vert R(u,v)\Vert_{\mathcal{N}_{p,q,r}^{\frac{n}{p}-2}}\leq C\Vert
u\Vert_{\dot{B}_{\infty,r_{1}}^{-1}}\Vert v\Vert_{\mathcal{N}_{p,q,r_{2}%
}^{\frac{n}{p}-1}}\leq C\Vert u\Vert_{\mathcal{N}_{p,q,r_{1}}^{\frac{n}{p}-1}%
}\Vert v\Vert_{\mathcal{N}_{p,q,r_{2}}^{\frac{n}{p}-1}}. \label{Prod.6}%
\end{equation}
But, if $(ii)$ holds, consider
\[
p^{\ast}:=\frac{2np}{3p+n}%
\]
and take $p^{^{\prime}}$ such that
\[
\frac{1}{p^{\ast}}=\frac{1}{p}+\frac{1}{p^{^{\prime}}},\quad\text{i.e.,}\quad
p^{^{\prime}}=\frac{2np}{3p-n}.
\]
Note that our hypotheses imply that $p^{\ast}\in\left[  \frac{2n}{5},\frac
{n}{2}\right)  ,\frac{2n}{5}>1$ and $p^{^{\prime}}\in(n,2n]$. Furthermore,
define
\[
q^{\ast}:=\frac{qp^{\ast}}{p}=\frac{2nq}{3p+n}%
\]
and take $q^{^{\prime}}$ such that
\[
\frac{1}{q^{\ast}}=\frac{1}{q}+\frac{1}{q^{^{\prime}}},\quad\text{i.e.,}\quad
q^{^{\prime}}=\frac{2nq}{3p-n}.
\]
We point out that $q^{^{\prime}}>q^{\ast}\geq1$, since $q\geq\frac{3p+n}{2n}$,
and $q\leq p$ implies that $q^{\ast}\leq p^{\ast}$ and $q^{^{\prime}}\leq
p^{^{\prime}}$. Having made these preparations, note that by using Lemma
\ref{embeddings} with $\theta=\frac{p^{\ast}}{p}$, we obtain the continuous
embedding
\[
\mathcal{N}_{p^{\ast},q^{\ast},r}^{\frac{n}{p^{\ast}}-2}\hookrightarrow
\mathcal{N}_{p,q,r}^{\frac{n}{p}-2}.
\]
So, there exists a constant $C>0$ such that
\begin{equation}
\Vert R(u,v)\Vert_{\mathcal{N}_{p,q,r}^{\frac{n}{p}-2}}\leq C\Vert
R(u,v)\Vert_{\mathcal{N}_{p^{\ast},q^{\ast},r}^{\frac{n}{p^{\ast}}-2}}.
\label{embedding1}%
\end{equation}
Then, as $\frac{2n}{5}\leq p^{\ast}<\frac{n}{2}$, we can make use of Lemma
\ref{R-BM-BM-Holder} to obtain the estimate
\begin{equation}
\Vert R(u,v)\Vert_{\mathcal{N}_{p^{\ast},q^{\ast},r}^{\frac{n}{p^{\ast}}-2}%
}\leq C\Vert u\Vert_{\mathcal{N}_{p,q,r_{1}}^{\frac{n}{p}-1}}\Vert
v\Vert_{\mathcal{N}_{p^{^{\prime}},q^{^{\prime}},r_{2}}^{\frac{n}{p^{^{\prime
}}}-1}}. \label{Prod.7}%
\end{equation}
Now, applying Lemma \ref{embeddings} with $\theta=\frac{p}{p^{^{\prime}}}$ and
drawing attention to the fact that $\frac{qp^{^{\prime}}}{p}=q^{^{\prime}}$,
we have the continuous embedding
\[
\mathcal{N}_{p,q,r_{2}}^{\frac{n}{p}-1}\hookrightarrow\mathcal{N}%
_{p^{^{\prime}},q^{^{\prime}},r_{2}}^{\frac{n}{p^{^{\prime}}}-1},
\]
which results in the estimate
\begin{equation}
\Vert v\Vert_{\mathcal{N}_{p^{^{\prime}},q^{^{\prime}},r_{2}}^{\frac
{n}{p^{^{\prime}}}-1}}\leq C\Vert v\Vert_{\mathcal{N}_{p,q,r_{2}}^{\frac{n}%
{p}-1}}. \label{embedding2}%
\end{equation}
Therefore, combining estimates (\ref{embedding1}), (\ref{Prod.7}) and
(\ref{embedding2}), we have
\begin{equation}
\Vert R(u,v)\Vert_{\mathcal{N}_{p,q,r}^{\frac{n}{p}-2}}\leq C\Vert
R(u,v)\Vert_{\mathcal{N}_{p^{\ast},q^{\ast},r}^{\frac{n}{p^{\ast}}-2}}\leq
C\Vert u\Vert_{\mathcal{N}_{p,q,r_{1}}^{\frac{n}{p}-1}}\Vert v\Vert
_{\mathcal{N}_{p^{^{\prime}},q^{^{\prime}},r_{2}}^{\frac{n}{p^{^{\prime}}}-1}%
}\leq C\Vert u\Vert_{\mathcal{N}_{p,q,r_{1}}^{\frac{n}{p}-1}}\Vert
v\Vert_{\mathcal{N}_{p,q,r_{2}}^{\frac{n}{p}-1}}. \label{Prod.8}%
\end{equation}
Thus, the result follows from Bony's decomposition (\ref{Bony}) and estimates
(\ref{Prod.4}), (\ref{Prod.5}), (\ref{Prod.6}) and (\ref{Prod.8}).

\begin{flushright}
\ding{110}
\end{flushright}

\section{Well-posedness of the stationary Hall-MHD system}

\label{Sec4}This section centers around the proof of Theorem \ref{Result1}, so
all the developed content and results used will be within three-dimensional
setting (i.e., for $n=3$). The core idea is to prove the well-posedness of the
extended system (\ref{Ext.HMHD1})-(\ref{Ext.HMHD4}) in suitable spaces and, as
will be detailed later, to derive from that the well-posedness of the
stationary Hall-MHD system (\ref{MHD}) from $Y$ to $X$ (see (\ref{X-space}%
)-(\ref{Y-space})). To effectively formulate the extended system and make it
amenable to Lemma \ref{auxiliar1}, we must employ certain vector identities.
In this sense, let us recall that, for any pair of divergence-free vector
fields $v,w\in\mathbb{R}^{3}$, we have%

\begin{equation}
(w\cdot\nabla)v=\dv(v\otimes w) \label{v-grad-w}%
\end{equation}
and%
\begin{equation}
(\nabla\times w)\times w=(w\cdot\nabla)w-\nabla\left(  \frac{|w|^{2}}%
{2}\right)  . \label{rotwxw}%
\end{equation}
So, denoting $\varphi:=\phi+\displaystyle\frac{|B|^{2}}{2}$, we have by
identities (\ref{v-grad-w}) and (\ref{rotwxw}) that equation (\ref{Ext.HMHD1})
recasts in
\begin{equation}
-\mu\Delta u+\nabla\varphi+\dv(u\otimes u)-\dv(B\otimes B)=f_{1}.
\label{use-leray}%
\end{equation}
Furthermore, the identity
\begin{equation}
\nabla\times(w\times v)=(v\cdot\nabla)w-(w\cdot\nabla)v, \label{gradx(wxv)-V1}%
\end{equation}
coupled with (\ref{v-grad-w}), results in
\begin{equation}
\nabla\times(w\times v)=\dv(w\otimes v)-\dv(v\otimes w). \label{gradx(wxv)-V2}%
\end{equation}
Then, employing the identity (\ref{gradx(wxv)-V2}) in the context of equations
(\ref{Ext.HMHD2}) and (\ref{Ext.HMHD3}), projecting equation (\ref{use-leray})
by means of the Leray projector $\mathcal{P}$, and taking $f_{3}:=\nabla\times
f_{2}$, the extended stationary Hall-MHD system (\ref{Ext.HMHD1}%
)-(\ref{Ext.HMHD4}) rewrites in
\begin{align}
-\mu\Delta u  &  =\mathcal{P}f_{1}+\beta_{1}(B,B)-\beta_{1}(u,u),\label{MHDE1}%
\\
-\nu\Delta B  &  =\mathcal{P}f_{2}+\beta_{2}(B,hJ-u),\label{MHDE2}\\
-\nu\Delta J  &  =\mathcal{P}f_{3}+\nabla\times\beta_{2}(curl^{-1}J,hJ-u),
\label{MHDE3}%
\end{align}
where
\[
\beta_{1}(v,w)=\frac{1}{2}\mathcal{P}(\dv(v\otimes w)+\dv(w\otimes v))
\]
and
\[
\beta_{2}(v,w)=\dv(v\otimes w)-\dv(w\otimes v).
\]
Setting $\Theta:=(u,B,J)$ and $\digamma:=(f_{1},f_{2},f_{3})$, system
(\ref{MHDE1})-(\ref{MHDE3}) can be written as
\begin{equation}
\Theta=\mathcal{L}\digamma+\mathcal{B}(\Theta,\Theta), \label{System1}%
\end{equation}
where $\mathcal{L}\digamma=%
\begin{pmatrix}
(-\mu\Delta)^{-1}\mathcal{P}f_{1}\\
(-\nu\Delta)^{-1}\mathcal{P}f_{2}\\
(-\nu\Delta)^{-1}\mathcal{P}f_{3}%
\end{pmatrix}
$ and $\mathcal{B}:\mathbb{R}^{3}\times\mathbb{R}^{3}\longrightarrow
\mathbb{R}^{9}$ is defined as
\[
\mathcal{B}(\Phi,\Psi)=%
\begin{pmatrix}
(-\mu\Delta)^{-1}(\beta_{1}(\Phi_{2},\Psi_{2})-\beta_{1}(\Phi_{1},\Psi_{1}))\\
(-\nu\Delta)^{-1}\beta_{2}(\Phi_{2},h\Psi_{3}-\Psi_{1})\\
(-\nu\Delta)^{-1}\nabla\times\beta_{2}(curl^{-1}\Phi_{3},h\Psi_{3}-\Psi_{1})
\end{pmatrix}
,
\]
in which $\Phi=(\Phi_{1},\Phi_{2},\Phi_{3})$ and $\Psi=(\Psi_{1},\Psi_{2}%
,\Psi_{3})$.\newline

Opting for the extended system instead of the initial one provides a benefit
due to its semi-linear nature, which stands in contrast with the
quasi-linearity seen in the Hall-MHD system for $(u,B)$. While the quadratic
terms in equations (\ref{MHDE1}) and (\ref{MHDE2}) bear a strong resemblance
to those found in the incompressible Navier-Stokes equations, the presence of
the Hall term in equation (\ref{MHDE3}) requires moving away from conventional
Navier-Stokes theory. This shift is necessary because, in this scenario,
differentiation occurs outside the first variable of $\beta_{2}$, rather than
within it.

Also, note that, under the assumptions of Theorem \ref{Result1}, if system
(\ref{MHDE1})-(\ref{MHDE3}) is well-posed from
\[
\tilde{Y}=\left\{  (f_{1},f_{2},f_{3})\in\mathcal{N}_{p,q,1}^{\frac{3}{p}%
-3}\times\mathcal{N}_{p,q,r}^{\frac{3}{p}-3}\times\mathcal{N}_{p,q,1}%
^{\frac{3}{p}-3} : \dv f_{2}=\dv f_{3}=0\text{ in}\,\,\mathcal{S}^{^{\prime}%
}(\mathbb{R}^{3})\right\}  ,
\]
endowed with the norm
\[
\Vert(f_{1},f_{2},f_{3})\Vert_{\tilde{Y}}=\Vert f_{1}\Vert_{\mathcal{N}%
_{p,q,1}^{\frac{3}{p}-3}}+\Vert f_{2}\Vert_{\mathcal{N}_{p,q,r}^{\frac{3}%
{p}-3}}+\Vert f_{3}\Vert_{\mathcal{N}_{p,q,1}^{\frac{3}{p}-3}},
\]
to
\[
\tilde{X}=\left\{  (u,B,J)\in\mathcal{N}_{p,q,1}^{\frac{3}{p}-1}%
\times\mathcal{N}_{p,q,r}^{\frac{3}{p}-1}\times\mathcal{N}_{p,q,1}^{\frac
{3}{p}-1}: \dv u=\dv B=\dv J=0\text{ in}\,\,\mathcal{S}^{^{\prime}}%
(\mathbb{R}^{3})\right\}  ,
\]
endowed with the norm
\[
\Vert(u,B,J)\Vert_{\tilde{X}}=\Vert u\Vert_{\mathcal{N}_{p,q,1}^{\frac{3}%
{p}-1}}+\Vert B\Vert_{\mathcal{N}_{p,q,r}^{\frac{3}{p}-1}}+\Vert
J\Vert_{\mathcal{N}_{p,q,1}^{\frac{3}{p}-1}},
\]
then, system (\ref{MHD}) is well-posed from $Y$ to $X$. Indeed, if $(u,B,J)$
is solution of (\ref{MHDE1})-(\ref{MHDE3}) and $f_{3}=\nabla\times f_{2}$, we
have that
\[
\nabla\times B-J=(-\nu\Delta)^{-1}\nabla\times\beta_{2}(curl^{-1}(\nabla\times
B - J),hJ-u).
\]
Therefore, Lemma \ref{Mikhlin-Hormander} and Proposition \ref{Product:3/p-1}
imply that
\begin{align*}
\Vert\nabla\times B-J\Vert_{\mathcal{N}_{p,q,1}^{\frac{3}{p}-1}}  &  =
\Vert(-\nu\Delta)^{-1}\nabla\times\beta_{2}(curl^{-1}(\nabla\times B -
J),hJ-u)\Vert_{\mathcal{N}_{p,q,1}^{\frac{3}{p}-1}}\\
&  \leq C \Vert\beta_{2}(curl^{-1}(\nabla\times B - J),hJ-u)\Vert
_{\mathcal{N}_{p,q,1}^{\frac{3}{p}-2}}\\
&  \leq C\Vert(curl^{-1}(\nabla\times B - J)\otimes(hJ-u)) -((hJ-u)\otimes
curl^{-1}(\nabla\times B - J))\Vert_{\mathcal{N}_{p,q,1}^{\frac{3}{p}-1}}\\
&  \leq C\Vert hJ-u\Vert_{\mathcal{N}_{p,q,1}^{\frac{3}{p}-1}}\Vert
curl^{-1}(\nabla\times B - J)\Vert_{\mathcal{N}_{p,q,1}^{\frac{3}{p}}}\\
&  \leq C\Vert(u,B,J)\Vert_{\tilde{X}}\Vert\nabla\times B - J\Vert
_{\mathcal{N}_{p,q,1}^{\frac{3}{p}-1}}.
\end{align*}
Thus, for $\Vert(f_{1},f_{2},\nabla\times f_{2})\Vert_{\tilde{Y}}$ small
enough (and consequently $\Vert(u,B,J)\Vert_{\tilde{X}}$), we obtain from the
above estimate that $J=\nabla\times B$, meaning that $(u,B)$ is a solution of
(\ref{MHD}).

In that case, we now carry out the proof of Theorem \ref{Result1}, which,
given the argument established above, amounts to proving the well-posedness of
system (\ref{MHDE1})-(\ref{MHDE3}) from $\tilde{Y}$ to $\tilde{X}$.

\vspace{0,5cm}

\noindent{\textbf{Proof of Theorem \ref{Result1}.}} Let $1\leq q\leq
p<\infty,1\leq r\leq2$ and assume that
\[
p<\frac{3}{2}\quad\text{or}\quad\frac{3}{2}\leq p<3\quad\text{and}\quad
q\geq\frac{p+1}{2}.
\]
By Lemma \ref{Mikhlin-Hormander}, we have that
\begin{align}
\Vert\mathcal{L}\digamma\Vert_{\tilde{X}}  &  =\Vert(-\mu\Delta)^{-1}%
\mathcal{P}f_{1}\Vert_{\mathcal{N}_{p,q,1}^{\frac{3}{p}-1}}+\Vert(-\nu
\Delta)^{-1}\mathcal{P}f_{2}\Vert_{\mathcal{N}_{p,q,r}^{\frac{3}{p}-1}}%
+\Vert(-\nu\Delta)^{-1}\mathcal{P}f_{3}\Vert_{\mathcal{N}_{p,q,1}^{\frac{3}%
{p}-1}}\nonumber\\
&  \leq C\left\{  \Vert f_{1}\Vert_{\mathcal{N}_{p,q,1}^{\frac{3}{p}-3}}+\Vert
f_{2}\Vert_{\mathcal{N}_{p,q,r}^{\frac{3}{p}-3}}+\Vert f_{3}\Vert
_{\mathcal{N}_{p,q,1}^{\frac{3}{p}-3}}\right\} \nonumber\\
&  \leq C\Vert\digamma\Vert_{\tilde{Y}}. \label{L-continuity}%
\end{align}
Furthermore, we know that {\fontsize{10}{10}\selectfont%
\begin{align}
\Vert\mathcal{B}(\Phi,\Psi)\Vert_{\tilde{X}}  &  =\Vert(-\mu\Delta)^{-1}%
(\beta_{1}(\Phi_{2},\Psi_{2})-\beta_{1}(\Phi_{1},\Psi_{1}))\Vert
_{\mathcal{N}_{p,q,1}^{\frac{3}{p}-1}}\nonumber\\
&  +\Vert(-\nu\Delta)^{-1}\beta_{2}(\Phi_{2},h\Psi_{3}-\Psi_{1})\Vert
_{\mathcal{N}_{p,q,r}^{\frac{3}{p}-1}}+\Vert(-\nu\Delta)^{-1}\nabla\times
\beta_{2}(curl^{-1}\Phi_{3},h\Psi_{3}-\Psi_{1})\Vert_{\mathcal{N}%
_{p,q,1}^{\frac{3}{p}-1}}. \label{Norm-B}%
\end{align}
} So let us estimate each term above separately. First, by Lemma
\ref{Mikhlin-Hormander} and Proposition \ref{Product:3/p-2}, it follows that%

\begin{align}
&  \Vert(-\mu\Delta)^{-1}(\beta_{1}(\Phi_{2},\Psi_{2})-\beta_{1}(\Phi_{1}%
,\Psi_{1}))\Vert_{\mathcal{N}_{p,q,1}^{\frac{3}{p}-1}}\nonumber\\
&  \leq C\left\{  \Vert\beta_{1}(\Phi_{2},\Psi_{2})\Vert_{\mathcal{N}%
_{p,q,1}^{\frac{3}{p}-3}}+\Vert\beta_{1}(\Phi_{1},\Psi_{1})\Vert
_{\mathcal{N}_{p,q,1}^{\frac{3}{p}-3}}\right\} \nonumber\\
&  \leq C\left\{  \Vert\Phi_{2}\otimes\Psi_{2}+\Psi_{2}\otimes\Phi_{2}%
\Vert_{\mathcal{N}_{p,q,1}^{\frac{3}{p}-2}}+\Vert\Phi_{1}\otimes\Psi_{1}%
+\Psi_{1}\otimes\Phi_{1}\Vert_{\mathcal{N}_{p,q,1}^{\frac{3}{p}-2}}\right\}
\nonumber\\
&  \leq C\left\{  \Vert\Phi_{2}\Vert_{\mathcal{N}_{p,q,2}^{\frac{3}{p}-1}%
}\Vert\Psi_{2}\Vert_{\mathcal{N}_{p,q,2}^{\frac{3}{p}-1}}+\Vert\Phi_{1}%
\Vert_{\mathcal{N}_{p,q,2}^{\frac{3}{p}-1}}\Vert\Psi_{1}\Vert_{\mathcal{N}%
_{p,q,2}^{\frac{3}{p}-1}}\right\} \nonumber\\
&  \leq C\left\{  \Vert\Phi_{2}\Vert_{\mathcal{N}_{p,q,r}^{\frac{3}{p}-1}%
}\Vert\Psi_{2}\Vert_{\mathcal{N}_{p,q,r}^{\frac{3}{p}-1}}+\Vert\Phi_{1}%
\Vert_{\mathcal{N}_{p,q,1}^{\frac{3}{p}-1}}\Vert\Psi_{1}\Vert_{\mathcal{N}%
_{p,q,1}^{\frac{3}{p}-1}}\right\} \nonumber\\
&  \leq C\Vert\Phi\Vert_{\tilde{X}}\Vert\Psi\Vert_{\tilde{X}}. \label{term1}%
\end{align}
Now, applying again Lemma \ref{Mikhlin-Hormander} and Proposition
\ref{Product:3/p-2} , we obtain
\begin{align}
&  \Vert(-\nu\Delta)^{-1}\beta_{2}(\Phi_{2},h\Psi_{3}-\Psi_{1})\Vert
_{\mathcal{N}_{p,q,r}^{\frac{3}{p}-1}}\nonumber\\
&  \leq C\Vert\beta_{2}(\Phi_{2},h\Psi_{3}-\Psi_{1})\Vert_{\mathcal{N}%
_{p,q,r}^{\frac{3}{p}-3}}\nonumber\\
&  \leq C\left\{  \Vert\Phi_{2}\otimes(h\Psi_{3}-\Psi_{1})\Vert_{\mathcal{N}%
_{p,q,r}^{\frac{3}{p}-2}}+\Vert(h\Psi_{3}-\Psi_{1})\otimes\Phi_{2}%
\Vert_{\mathcal{N}_{p,q,r}^{\frac{3}{p}-2}}\right\} \nonumber\\
&  \leq C\Vert\Phi_{2}\Vert_{\mathcal{N}_{p,q,r}^{\frac{3}{p}-1}}(\Vert
\Psi_{3}\Vert_{\mathcal{N}_{p,q,\infty}^{\frac{3}{p}-1}}+\Vert\Psi_{1}%
\Vert_{\mathcal{N}_{p,q,\infty}^{\frac{3}{p}-1}})\nonumber\\
&  \leq C\Vert\Phi_{2}\Vert_{\mathcal{N}_{p,q,r}^{\frac{3}{p}-1}}(\Vert
\Psi_{3}\Vert_{\mathcal{N}_{p,q,1}^{\frac{3}{p}-1}}+\Vert\Psi_{1}%
\Vert_{\mathcal{N}_{p,q,1}^{\frac{3}{p}-1}})\nonumber\\
&  \leq C\Vert\Phi\Vert_{\tilde{X}}\Vert\Psi\Vert_{\tilde{X}}. \label{term2}%
\end{align}
Finally, Lemma \ref{Mikhlin-Hormander} and Proposition \ref{Product:3/p-1}
lead us to the estimate
\begin{align}
&  \Vert(-\nu\Delta)^{-1}\nabla\times\beta_{2}(curl^{-1}\Phi_{3},h\Psi
_{3}-\Psi_{1})\Vert_{\mathcal{N}_{p,q,1}^{\frac{3}{p}-1}}\nonumber\\
&  \leq C\Vert\beta_{2}(curl^{-1}\Phi_{3},h\Psi_{3}-\Psi_{1})\Vert
_{\mathcal{N}_{p,q,1}^{\frac{3}{p}-2}}\nonumber\\
&  \leq C\left\{  \Vert(curl^{-1}\Phi_{3})\otimes(h\Psi_{3}-\Psi_{1}%
)\Vert_{\mathcal{N}_{p,q,1}^{\frac{3}{p}-1}}+\Vert(h\Psi_{3}-\Psi_{1}%
)\otimes(curl^{-1}\Phi_{3})\Vert_{\mathcal{N}_{p,q,1}^{\frac{3}{p}-1}}\right\}
\nonumber\\
&  \leq C\Vert curl^{-1}\Phi_{3}\Vert_{\mathcal{N}_{p,q,1}^{\frac{3}{p}}%
}(\Vert\Psi_{3}\Vert_{\mathcal{N}_{p,q,1}^{\frac{3}{p}-1}}+\Vert\Psi_{1}%
\Vert_{\mathcal{N}_{p,q,1}^{\frac{3}{p}-1}})\nonumber\\
&  \leq C\Vert\Phi_{3}\Vert_{\mathcal{N}_{p,q,1}^{\frac{3}{p}-1}}(\Vert
\Psi_{3}\Vert_{\mathcal{N}_{p,q,1}^{\frac{3}{p}-1}}+\Vert\Psi_{1}%
\Vert_{\mathcal{N}_{p,q,1}^{\frac{3}{p}-1}})\nonumber\\
&  \leq C\Vert\Phi\Vert_{\tilde{X}}\Vert\Psi\Vert_{\tilde{X}}. \label{term3}%
\end{align}
Then, linking equality (\ref{Norm-B}) with estimates (\ref{term1}),
(\ref{term2}) and (\ref{term3}), it is clear that
\begin{equation}
\Vert\mathcal{B}(\Phi,\Psi)\Vert_{\tilde{X}}\leq C\Vert\Phi\Vert_{\tilde{X}%
}\Vert\Psi\Vert_{\tilde{X}}. \label{B-continuity}%
\end{equation}
Therefore, estimates (\ref{L-continuity}) and (\ref{B-continuity}) guarantee
that we are under the assumptions of Lemma \ref{auxiliar1}, which enables us
to apply it and conclude the result.

\begin{flushright}
\ding{110}
\end{flushright}

\section{Well-posedness of the stationary Navier-Stokes equations}

\label{Sec5} In this section, we focus on the proof of Theorem \ref{Result2},
which, as mentioned earlier, will be carried out with the aid of Lemma
\ref{auxiliar1} and the product estimate obtained through Proposition
\ref{Product:3/p-2}.

\vspace{0,5cm}

\noindent{\textbf{Proof of Theorem \ref{Result2}.}} Let $n\geq3$, $1\leq q\leq
p<\infty$, $1\leq r\leq\infty$, and suppose that
\[
p<\frac{n}{2}\quad\text{or}\quad\frac{n}{2}\leq p<n\quad\text{and}\quad
q\geq\frac{3p+n}{2n}.
\]
Applying the Leray projector $\mathcal{P}$ in $(\ref{aux-NS100})_{1}$, we
arrive at
\begin{equation}
-\mu\Delta u+\mathcal{P}\dv(u\otimes u)=\mathcal{P}f. \label{NS-projected}%
\end{equation}
Then, taking $\mathcal{L}:=(-\mu\Delta)^{-1}\mathcal{P}$ and $\mathcal{B}%
:=-(-\mu\Delta)^{-1}\mathcal{P}\dv(.\otimes.)$, equation (\ref{NS-projected})
is equivalent to
\[
u=\mathcal{L}f+\mathcal{B}(u,u).
\]
Now, note that, by means of Lemma \ref{Mikhlin-Hormander}, we have
\begin{equation}
\Vert\mathcal{L}f\Vert_{\mathcal{N}_{p,q,r}^{\frac{n}{p}-1}}=\Vert(-\mu
\Delta)^{-1}\mathcal{P}f\Vert_{\mathcal{N}_{p,q,r}^{\frac{n}{p}-1}}\leq C\Vert
f\Vert_{\mathcal{N}_{p,q,r}^{\frac{n}{p}-3}}. \label{L-continuity-NS}%
\end{equation}
Furthermore, Lemma \ref{Mikhlin-Hormander} and Proposition \ref{Product:3/p-2}
yield the estimate
\begin{align}
\Vert\mathcal{B}(v,w)\Vert_{\mathcal{N}_{p,q,r}^{\frac{n}{p}-1}}  &
=\Vert-(-\mu\Delta)^{-1}\mathcal{P}\dv(v\otimes w)\Vert_{\mathcal{N}%
_{p,q,r}^{\frac{n}{p}-1}}\nonumber\\
&  \leq C\Vert\dv(v\otimes w)\Vert_{\mathcal{N}_{p,q,r}^{\frac{n}{p}-3}%
}\nonumber\\
&  \leq C\Vert v\otimes w\Vert_{\mathcal{N}_{p,q,r}^{\frac{n}{p}-2}%
}\nonumber\\
&  \leq C\Vert v\Vert_{\mathcal{N}_{p,q,r}^{\frac{n}{p}-1}}\Vert
w\Vert_{\mathcal{N}_{p,q,\infty}^{\frac{n}{p}-1}}\nonumber\\
&  \leq C\Vert v\Vert_{\mathcal{N}_{p,q,r}^{\frac{n}{p}-1}}\Vert
w\Vert_{\mathcal{N}_{p,q,r}^{\frac{n}{p}-1}}. \label{B-continuity-NS}%
\end{align}
Hence, by estimates (\ref{L-continuity-NS}) and (\ref{B-continuity-NS}), we
have that $\mathcal{L}$ and $\mathcal{B}$ fulfill the requirements of Lemma
\ref{auxiliar1}. Consequently, equation (\ref{NS-projected}) is well-posed
from $\mathcal{N}_{p,q,r}^{\frac{n}{p}-3}$ to $\mathcal{N}_{p,q,r}^{\frac
{n}{p}-1}$, which proves the result.

\begin{flushright}
\ding{110}
\end{flushright}

\

\noindent\textbf{Acknowledgments.} LCFF was supported by CNPq (grant:
312484/2023-2), Brazil. RPS was supported by UTFPR, Brazil.

\thispagestyle{empty} \vspace{0.5cm}

\noindent\textbf{Conflict of interest statement.} The authors declare that
they have no conflict of interest.

\

\noindent\textbf{Data availability statement.} This manuscript has no
associated data.

\end{document}